\documentclass[10pt]{article}
 \usepackage{amsmath,amsfonts,amssymb,amsthm}
 \usepackage{hyperref}

 \theoremstyle{plain}
 \newtheorem{theorem}{Theorem}[section]
 \newtheorem{lemma}[theorem]{Lemma}
 \newtheorem{corollary}[theorem]{Corollary}
 \newtheorem{proposition}[theorem]{Proposition}

 \theoremstyle{definition}
 \newtheorem{definition}[theorem]{Definition}

\theoremstyle{example}
 \newtheorem*{example}{Example}

 \theoremstyle{remark}
 \newtheorem*{remark}{Remark}
 \newtheorem*{notation}{Notation}

 \title{Gravitational instantons with faster than quadratic curvature decay (I)}

 \author{Gao Chen and Xiuxiong Chen}

\begin{document}

 \maketitle

\tableofcontents

 \section{Introduction}

 A gravitaional instanton is a complete, hyperk\"ahler 4-manifold with curvature decaying fast enough at infinity. In 1977, gravitational instanton was first introduced by Hawking as building block of the Euclidean quantum gravity theory \cite{Hawking}. Even though physicists expect the faster than quadratic curvature decay at infinity,  this seems has not been made precise in literatures.

  This is our first paper in a series to study gravitational instantons. For clarity,  we always assume that the curvature satisfies a decay condition

  \begin{equation} |\mathrm{Rm}|(x) \le r(x)^{-2-\epsilon},\;\label{eq:faster than quadratic decay} \end{equation}

  where $r(x)$ denotes the metric distance to a base point $o$ in the complex surface and $\epsilon > 0$ is any small positive number, say $<\frac{1}{100}.\;$ Under those conditions, we want to study two fundamental questions:

  \begin{enumerate}
  \item The differential and metric structure of the infinity of these gravitational instantons. Note that this is different from the tangent cone at infinity, especially if the volume growth is sub-Euclidean.
  \item Given these end structures, to what extent, do we know these instantons globally and holomorphically? In other words, is gravitational instanton uniquely determined by its end structure?
  \end{enumerate}

   Both problems seem to be well known to the research community. Since 1977, many examples of gravitational instantons have been constructed \cite{Hawking} \cite{AtiyahHitchin} \cite{Kronheimer1} \cite{CherkisKapustin}. The end structures of these examples are completely known now. According to the volume growth rate, they can be divided into four categories:  ALE, ALF, ALG and ALH, where the volume growth are of order 4,3,2 and 1 respectively. For the convenience of readers, we will give a precise definition of these ends in Section 2. There is a folklore conjecture that when the curvature decay fast enough, any gravitational instantons must be asymptotic to one of the standard models of ends.

   In ALE case, we understand these instantons completely through the work of Kronheimer in \cite{Kronheimer1}  \cite{Kronheimer2}. In the remaining cases, the asymptotical volume growth rate is usually hard to control, and may oscillate and may even not be an integer. In an important paper, with additional assumption that the volume growth rate is sub-Euclidean but at least cubic and a slightly weaker curvature decay condition depending on volume growth rate,  Minerbe \cite{MinerbeMass} \cite{MinerbeAsymptotic} proved that it must be ALF.  In our paper, we first prove the folklore conjecture.

  \begin{theorem}
  (Main Theorem 1) Let $(M^4,g)$ be a connected complete hyperk\"ahler manifold with curvature decaying as (\ref{eq:faster than quadratic decay}), then it must be asymptotic to the standard metric of order $\epsilon$. Consequently, it must be one of the four families: ALE,ALF, ALG and ALH.
  \label{main-theorem-1}
  \end{theorem}

  For more detail about this theorem, see Theorem \ref{one-end}, Theorem \ref{ALE}, and Theorem \ref{standard-fiberation}.
  We would like to remark that the curvature condition can not be weaken to $|\mathrm{Rm}|=O(r^{-2})$.  In 2012, besides the study of ALG and ALH instantons on rational elliptic surfaces, Hein \cite{Hein} also constructed two new classes of hyperk\"ahler metrics on rational elliptic surfaces with volume growth, injective radius decay, and curvature decay rates $r^{4/3}$, $r^{-1/3}$, $r^{-2}$, and $r^2$, $(\log r)^{-1/2}$, $r^{-2}(\log r)^{-1}$, respectively. Note that  curvature does not satisfy (\ref{eq:faster than quadratic decay}) and  they do not belong to any of the four families!

  Our new contribution lies in ALG and ALH cases; in ALF case, our contribution is to remove the volume growth constraint from  Minerbe's work \cite{MinerbeMass}. In fact, Minerbe's volume growth constraint becomes an corollary instead condition of our first main theorem. In particular, we can now apply his work and improve the curvature decay rate of an ALF instanton to $O(r^{-3})$. Therefore, the asymptotic rate can be improved to any $\delta<1$. For ALH-non-splitting instantons, we can also improve the curvature decay rate. It turns out that the metric must converge to the flat model exponentially. For more details, see Theorem \ref{exponential-asymptotics}. We believe that there is a similar self improvement for ALG instantons, but we will leave it for future study.\\

  For the second question, the crucial point is to understand  ``the end" holomorphically. In ICM 1978, Yau conjectured that every complete Calabi-Yau manifold can be compactified in the complex analytic sense \cite{Yau}. There are counterexamples if we only assume the completeness without fast curvature decaying condition \cite{AndersonKronheimerLebrun}. However, when we assume the faster than quadratic curvature decay condition, in both ALG and ALH-non-splitting cases, we can prove Yau's conjecture. In higher dimension $n\geq 3$, assuming the curvature exponentially decay and the metric is asymptotically cylindrical, Haskins, Hein and Nordstr\"om \cite{HaskinsHeinNordstrom} constructed a compactification and therefore verified Yau's conjecture in their settings.

  \begin{theorem}
  (Main Theorem 2)
  For any ALG or ALH-non-splitting gravitational instanton $M$, there exist a compact elliptic surface $\bar M$ with a meromorphic function $u:\bar M\rightarrow\mathbb{CP}^1$ whose generic fiber is torus. The fiber $D=\{u=\infty\}$ is regular if $M$ is ALH, while it is either regular or of type I$_0^*$, II, II$^*$, III, III$^*$, IV, IV$^*$ if $M$ is ALG. There exist an $(a_1,a_2,a_3)$ in $\mathbb{S}^2$ such that when we use $a_1I+a_2J+a_3K$ as the complex structure, $M$ is biholomorphic to $\bar M-D$.
  \label{main-theorem-2}
  \end{theorem}

  The converse problem is actually very well known and has been studied actively: Given a compact complex manifold $\bar M$ and $D$ an anti-canonical  divisor, do we have complete Ricci-flat K\"ahler metric on $\bar M\setminus D$? Tian-Yau \cite{TianYau} proved that for quasi-projective surface $M=\bar{M}-D$ with $\bar{M}$ smooth and $D$ a smooth anticanonical divisor in $\bar{M}$, as long as $D^2\ge0$, $M = \bar M\setminus D$ has a complete Ricci-flat K\"ahler metric which has the volume growth of linear order (i.e ALH).  In \cite{Hein}, Hein  generalized Tian-Yau's work and constructed ALG gravitational instantons on the complement of the anticanonical divisor to the rational elliptic surface $\bar{M}\setminus D$. We do not know whether one can repeat the Hein-Tian-Yau construction for general elliptic surface which is not algebraic. The complete understanding of ALG and ALH-non-splitting gravitational instantons is even harder.\\

  In ALF case, more discussions are needed:
  \begin{enumerate}
  \item In ALF-$A_k$ case, Minerbe \cite{MinerbeMultiTaubNUT} proved that any ALF-$A_k$ instanton must be the trivial product or the multi-Taub-NUT metric. In particular, there is no ALF-$A_k$ instantons for $k<-1$.

  \item  In ALF-$D_k$ case, Biquard and Minerbe \cite{BiquardMinerbe} proved that there is no ALF-$D_k$ instantons for $k<0$. For $k\ge 0$, the first example was constructed by Atiyah and Hitchin \cite{AtiyahHitchin}, where $k=0$. Ivanov and Ro\v{c}ek \cite{IvanovRocek} conjectured a formula for larger $k$ using generalized Legendre transform developed by Lindstr\"om and Ro\v{c}ek. This conjecture was proved by Cherkis and Kapustin \cite{CherkisKapustin} and computed more explicitly by Cherkis and Hitchin \cite{CherkisHitchin}. It is conjectured that any ALF-$D_k$ instanton must be exactly the metric constructed by them. This conjecture has not been solved yet. However we are able to prove the existence of the $\mathcal{O}(4)$ multiplet which plays an important role in the Cherkis-Hitchin-Kapustin-Ivanov-Lindstr\"om-Ro\v{c}ek construction.
 \end{enumerate}

  \begin{theorem}
  (Main Theorem 3)
  In the ALF-$D_k$ case, there exists a holomorphic map from the twistor space of $M$ to the total space of the $\mathcal{O}(4)$ bundle over $\mathbb{CP}^1$ which commutes with both the projection to $\mathbb{CP}^1$ and the real structure.
  \label{holomorphic-map-on-twistor-space}
  \end{theorem}

  For the definitions of twistor space and the real structure, see Theorem \ref{twistor-space} and Theorem \ref{holomorphic-map-on-twistor-space}.

  One of our main tools comes from the equivalence between the hyperk\"ahler condition and the Calabi-Yau condition. Actually, for hyperk\"ahler manifolds, we have three complex structures $I, J, K$. They induce three symplectic forms by $$\omega_1(X,Y)=g(IX,Y), \omega_2(X,Y)=g(JX,Y),\omega_3(X,Y)=g(KX,Y).$$ The form $\omega^+=\omega_2+i\omega_3$ is a $I$-holomorphic symplectic form. This induces the equivalence of $\mathrm{Sp}(1)$ and $\mathrm{SU}(2)$. Notice that for any $(a_1,a_2,a_3)\in\mathbb{S}^2$, $a_1 I+a_2 J+a_3 K$ is a K\"ahler structure. There is a special property of $\mathrm{Sp}(1)$: Given any vectors $v, w \in T_p$ which are orthogonal to each other and have same length, there exists an $(a_1,a_2,a_3)$ in $\mathbb{S}^2$ such that $(a_1I+a_2J+a_3K) v = w$. We will use this property to find the best complex structure.

  We obviously benefit from studying a series of papers by Minerbe \cite{MinerbeMass}, \cite{MinerbeAsymptotic}, and \cite{MinerbeMultiTaubNUT}. Although his work seems only valid in ALF-$A_k$ case,
  we manage to make some modest progress in all cases in the present work.\\

  \noindent{\bf Acknowledgement:}  Both authors are grateful to the insightful and helpful discussions with Sir Simon Donaldson, Blaine Lawson, Claude LeBrun and Martin Ro\v{c}ek. We also thank Gilles Carron and Yu Li for some suggestions of improvements on the first version of this paper.

 \section{Notations and definitions}

 First, let us understand the standard models near infinity. The explicit expression of those models are defined in Theorem \ref{standard-fiberation}. To avoid singularity, a ball $B_R$ is always removed.

 \begin{example} Let $(X, h_1)$ be any manifold of dimension $3-k$ with constant sectional curvature 1 and $C(X)$ its metric cone with standard flat metric $\mathrm{d}r^2 + r^2 h_1.\;$
 Let $\mathbb{T}^k$ be a $k$-dimensional flat torus. Then the $\mathbb{T}^k$ fibration $E$ over $C(X)-B_R$ with a $\mathbb{T}^k$ invariant metric $h$ provides the standard model near infinity.
 \begin{enumerate}
\item $C(X)=\mathbb{R}^4/\Gamma$, $\Gamma$ is a discrete subgroup in $\mathrm{SU}(2)\;$ acting freely on $\mathbb{S}^3$. In this case, $(E,h)=C(X)-B_R$ with the flat metric. It is called ALE.
\item
 $C(X)=\mathbb{R}^3$, $(E,h)$ is either the trivial product $(\mathbb{R}^3-B_R)\times\mathbb{S}^1$ or the quotient of the Taub-NUT metric with mass $m$ outside a ball by $\mathbb{Z}_{|e|}$, where $me<0$. It is called ALF-$A_k$ with $k=-1$ in the first case and $k=-e-1$ in the second case.
\item
  $C(X)=\mathbb{R}^3/\mathbb{Z}_2$, $(E,h)$ is either the $\mathbb{Z}_2$ quotient of the trivial product of $\mathbb{R}^3-B_R$ and $\mathbb{S}^1$ or the quotient of the Taub-NUT metric with mass $m$ outside a ball by the binary dihedral group $D_{4|e|}$ of order $4|e|$, where $me<0$. It is called ALF-$D_k$ with $k=2$ for the first case, and $k=-e+2$ for the second case.
\item
  $C(X)$ is the flat cone $\mathbb{C}_{\beta}$ with cone angle $2\pi\beta$, $(E,h)$ is a torus bundle over $\mathbb{C}_{\beta}-B_R$ with a flat metric, where $(\beta,E,h)$ are in the list of some special values; It is called ALG.
\item
  $C(X)=\mathbb{R}$, $(E,h)$ is the product of $\mathbb{R}-B_R$ and a flat 3-torus. It is called ALH-splitting.
\item
$C(X)=\mathbb{R}_+$, $(E,h)$ is the product of $[R,+\infty)$ and a flat 3-torus. It is called ALH-non-splitting.
 \end{enumerate}
 \end{example}

 We may call such fiberation a standard model near infinity. It serves as an asymptotic model in the following sense:

 \begin{definition}  A complete Riemannian manifold  $(M, g)$ is called asymptotic to the standard model $(E,h)$ of order $\delta$ if there exist a bounded domain $K \subset M$, and a diffeomorphism $\Phi:  E\rightarrow M \setminus K$ such that
 $$\Phi^* g =  h + O'(r^{-\delta})$$
 for some $\delta > 0.\;$
 \end{definition}

 Any manifold asymptotic to the standard ALE model is called ALE. It stands for asymptotically locally Euclidean.
 Similarly, any manifold asymptotic to the standard ALF model is called ALF. It means asymptotically locally flat.
 The ALG and ALH manifold are defined similarly. The letters ``G" and ``H" do not have any meanings. They are just the letters after ``E" and ``F".

 Notice that our definition of ALH manifold is different from the definition of Hein in \cite{Hein}. However, Theorem \ref{exponential-asymptotics} implies that there is no essential difference for gravitational instantons.

 \begin{notation}
 $o$ is a fixed point in $M$. In Section 3, $r(p)=\mathrm{dist}(o,p)$ is the geodesic distance between $o$ and $p$. In Section 4, $E$ is a fiberation over $C(X)-B_R=\{(r,\theta):r\ge R,\theta\in X\}$. So the pull back of $r$ by the projection is a function on $E$. On $M$, we pull back that function, cut it off by some smooth function, and add 1 to get a smooth function $r\ge 1$. The reader should be careful about the switch of the meanings of $r$ in different sections of our paper.

 $O'(r^\alpha)$ means that for any $m\ge 0$, the $m$-th derivative of the tensor belongs to $O(r^{\alpha-m})$.
 $\chi$ will be a smooth cut-off function from $(-\infty,+\infty)$ to $[0,1]$ such that $\chi\equiv 1$ on $(-\infty,1]$ and $\chi\equiv 0$ on $[2,\infty)$. We will always use $\Delta=-\nabla^*\nabla$ as the Laplacian operator.
 \end{notation}

\section{Asymptotic Fibration}

In this section, we prove the main theorem 1. It is essentially a theorem in Riemannian geometry. The basic tool is to view a ball in the manifold $M$ as a quotient of the ball inside the tangent space equipped with the metric pulled back from exponential map by the group of local covering transforms which correspond to the short geodesic loops in $M$. In the first subsection, we discuss this picture. In the second subsection, we provide a rough estimate of the holonomy of short geodesic loops. In the third subsection, we use that rough estimate to classify the tangent cone at infinity. In the fourth subsection, we use this information to get a better control of geodesic loops. Finally, we use this better control to prove our main theorem 1.

\subsection{Short geodesic loops and the local covering space}

In 1978 Gromov \cite{Gromov} started the research of almost flat manifolds, i.e. manifold with very small curvature. In 1981, Buser and Karcher wrote a book \cite{BuserKarcher} to explain the ideas of Gromov in detail. In 1982 Ruh \cite{Ruh} gave a new way to understand it. Assume $p$ is a point in $M$. The exponential map $\exp:T_p\rightarrow M$ is a local covering map inside the conjugate radius. We can pull back the metric from $M$ using the exponential map inside conjugate radius. There is a lemma about the local geometry on the tangent space:

\begin{lemma}
Suppose $g_{ij}$ is a metric on $B_1(0)\subset \mathbb{R}^{n}$ satisfying the following condition:

(1) The curvature is bounded by $\Lambda^2$;

(2) $g_{ij}(0)=\delta_{ij}$;

(3) The line $\gamma(t)=t\mathbf{u}$ is always a geodesic for any unit vector $\mathbf{u}$.

Then there exist constants $\Lambda(n)$ and $C(m,n)$ such that as long as $\Lambda\le\Lambda(n)<\pi/2$,

(1) Any two points $x$ and $y$ in $B_1(0)$ can be connected by a unique minimal geodesic inside $B_1(0)$;

(2) If the Ricci curvature is identically 0, then $|D^m(g_{ij}(x)-\delta_{ij})|< C(m,n)\Lambda^2$ for all $m\ge 0$ and $x\in B_{1/2}$.

\label{local-geometry}
\end{lemma}

\begin{proof}
(1) It was proved by Buser and Karcher as the Proposition 6.4.6 in \cite{BuserKarcher}.

(2) Therefore, all the works in \cite{JostKarcher} apply. We can find functions $l_i$ satisfying $|\nabla l_i(x)-e_i(x)|\le C(n)\Lambda^2$ and $|\nabla^2 l_i(x)|\le C(n)\Lambda^2$ for all $x\in B_{1/2}(0)$ as long as $\Lambda(n)$ is small enough, where $e_i(x)$ is a vector field which is parallel along radical geodesics and
equals to $\frac{\partial}{\partial x_i}$ at origin.
For even smaller $\Lambda(n)$, we can use $l_i$ as coordinate functions in $L_{0.9}(0)=\{\sum l_i^2<(0.9)^2\}\subset B_1(0)=\{\sum x_i^2<1\}$. In this coordinate
$$|g_{ij}w^iw^j-|w|^2|\le C(n)\Lambda^2|w|^2\le 0.01 |w|^2,$$
$$|\partial_k g_{ij}|< C(n)\Lambda^2<1,$$
$$\Delta u=\frac{1}{\sqrt{G}}\partial_j(\sqrt{G}g^{ij}\frac{\partial u}{\partial l_i}).$$
What is more $|\Delta{l_i}|< C(n)\Lambda^2$. By Theorem 9.15 of \cite{GilbargTrudinger}, for all $1<p<\infty$, there is a unique solution $u_i\in W^{2,p}(L_{0.9})\cap W_{0}^{1,p}(L_{0.9})$ such that $\Delta u_i=\Delta l_i$.
By Lemma 9.17 of \cite{GilbargTrudinger}, we actually have
$$||u_i||_{W^{2,p}(L_{0.9}(0))}< C(n,p)||\Delta l_i||_{L^p(L_{0.9}(0))}< C(n,p)\Lambda^2.$$
By Sobelev embedding theorem (c.f. Theorem 7.26 of \cite{GilbargTrudinger}),
$$||u_i||_{C^1(\overline{L_{0.9}(0)})}< C(n)||u_i||_{W^{2,2n}(L_{0.9}(0))}<C(n)\Lambda^2.$$
In particular, when $\Lambda(n)$ is small enough, $h_i=l_i-u_i$ gives a harmonic coordinate in $H_{0.8}(0):=
\{\sum h_i^2<(0.8)^2\}\subset L_{0.9}(0)$.
In this harmonic coordinate, $1/1.02|w|^2<g_{ij}w^{i}w^{j}<1.02|w|^2.$
By elliptic regularity, actually all the above functions are smooth. So we can differentiate them to get equations.
Since $\Gamma_{ij}^{k}g^{ij}=0$, we know that $2\mathrm{Ric}_{mk}=g^{im}R_{ijkl}g^{jl}+g^{ik}R_{ijml}g^{jl}$ satisfies
$$g^{rs}\frac{\partial^2 (g_{ij}-\delta_{ij})}{\partial h_r\partial h_s}=-2\mathrm{Ric}_{ij}+Q_{ij}(g,\partial g)+Q_{ji}(g,\partial g),$$
where $$Q_{mk}(g,\partial g)=g^{jl}\partial_{l}g_{im}\Gamma^{i}_{kj}
-g^{jl}g_{im}\Gamma^{h}_{kj}\Gamma^{i}_{lh}-g_{im}\partial_{k}g^{jl}\Gamma^{i}_{jl}.$$
We already know that $||g_{ij}-\delta_{ij}||_{W^{1,p}(H_{0.8}(0))}<C(n)\Lambda^2$ from the $W^{2,p}$ bound of $u_i$.
So $||Q_{ij}(g,\partial g)||_{L^{p/2}(H_{0.8}(0))}<C(n)\Lambda^4$. When the Ricci curvature is identically 0, by Theorem 9.11 of \cite{GilbargTrudinger}, we have
$$||g_{ij}-\delta_{ij}||_{W^{2,p/2}(H_{0.7})}< C(n)(||g_{ij}-\delta_{ij}||_{L^{p/2}(H_{0.8})}+||Q_{ij}||_{L^{p/2}(H_{0.8})})<C(n)\Lambda^2.$$
After taking more derivatives, we can get the required bound in the harmonic coordinate. This in turn bounds the Christoffel symbol and gives a bound of the geodesic equation. So when we solve this geodesic equation, we can get the required bound in the geodesic ball.
\end{proof}

The above estimate is an interior estimate. The number $1/2$ can be replaced by any number smaller than 1.

To find out the local covering transform, we look at the preimage $p_1$ of $p$ under the exponential map inside $B_1(0)$. There is a local covering transform $F$ which maps $0$ to $p_1$. The image of the radical geodesic from $0$ to $p_1$ is a geodesic loop based at $p$. This gives a 1-1 correspondence between short geodesic loops and covering transforms.

Now suppose we have two short enough geodesic loops $\gamma_1$ and $\gamma_2$ with same base point $p$. Then they correspond to two local covering transforms $F_1$ and $F_2$. The composition $F_1\circ F_2$ is also a local covering transform. It corresponds to another geodesic loop based at $p$. It is exactly the product of $\gamma_1$ and $\gamma_2$ defined by Gromov.

For any $q$ close enough to $p$, choose an preimage $q_0$ of $q$ close enough to $0$, then $q_1=F(q_0)$ is another preimage of $q$ which is very close to $p_1$.
The image of the shortest geodesic connecting $q_1$ and $q_2$ under the exponential map is a geodesic loop based at $q$. It is called the sliding of $\gamma$. When $q$ moves along a curve $\alpha$, the sliding of $\gamma$ becomes a 1-parameter family of curves. It is called the sliding of $\gamma$ along the curve $\alpha$.

When we parallel transport any vector $v$ along the geodesic loop, we will get another vector $\mathrm{hol}(v)$. $\mathrm{hol}:T_p\rightarrow T_p$ is called the holonomy of the loop. For hyperk\"ahler manifold, $\mathrm{hol}\in\mathrm{Sp}(1)=\mathrm{SU}(2)$. Under suitable orthonormal basis, any element in $\mathrm{SU}(2)$ can be written as
 \[
 \mathbf{A}=
 \left( {\begin{array}{*{20}c}
    e^{i\theta} & 0  \\
    0 & e^{-i\theta}   \\
 \end{array}} \right)
 .
 \]
 So
 \[
 \mathbf{A}-\mathbf{Id}=
 \left( {\begin{array}{*{20}c}
    e^{i\theta}-1 & 0  \\
    0 & e^{-i\theta}-1   \\
 \end{array}} \right),
 (\mathbf{A}-\mathbf{Id})
 \left( {\begin{array}{*{20}c}
    v_1  \\
    v_2  \\
 \end{array}} \right)
 =
 \left( {\begin{array}{*{20}c}
    (e^{i\theta}-1)v_1  \\
    (e^{-i\theta}-1)v_2  \\
 \end{array}} \right).
 \]
 So $|(\mathbf{A}-\mathbf{Id})\mathbf{v}|=|\mathbf{A}-\mathbf{Id}||\mathbf{v}|$ if we define the norm by $$|\mathbf{A}-\mathbf{Id}|=|e^{i\theta}-1|=|e^{-i\theta}-1|.$$
 This property is also a special property of $\mathrm{SU}(2)$. For instance $\mathrm{SO}(4)$ does not have this property.

 In the flat case, local covering transforms are all linear maps. Suppose $T_1(\mathbf{x})=\mathbf{ax}+\mathbf{b}$, $T_2(\mathbf{x})=\mathbf{Ax}+\mathbf{B}$ are two local covering transforms, where $\mathbf{a},\mathbf{A}\in\mathrm{SO}(n)$  and $ \mathbf{b},\mathbf{B}\in\mathbb{R}^n$.
 They correspond to two geodesic loops $\gamma_1, \gamma_2$ with same base point $p$. $\mathbf{A}$, $\mathbf{a}$ are exactly the holonomy of $\gamma_1$ and $\gamma_2$ while $|\mathbf{B}|$, $|\mathbf{b}|$ are the same as the length of loops $\gamma_1$ and $\gamma_2$ respectively.
 $$T_1\circ T_2(\mathbf{x})=\mathbf{a}(\mathbf{Ax}+\mathbf{B})+\mathbf{b}=\mathbf{aAx}+\mathbf{aB}+\mathbf{b}$$ will correspond to the Gromov product of $\gamma_1$ and $\gamma_2$.
 So $$T_1^{-1}T_2^{-1}T_1T_2(\mathbf{x})=\mathbf{a}^{-1}\mathbf{A}^{-1}\mathbf{aAx}+
 \mathbf{a}^{-1}\mathbf{A}^{-1}((\mathbf{a}-\mathbf{Id})\mathbf{B}+(\mathbf{Id}-\mathbf{A})\mathbf{b}).$$
 The Lie algebra are also linear maps.
 Taking the derivative in the above expression of the commutator at the origin $T_1(\mathbf{x})=T_2(\mathbf{x})=\mathrm{Id}(\mathbf{x})=\mathrm{Id}(\mathbf{x})+\mathbf{0}$, the Lie bracket is
 $$[\mathbf{ax}+\mathbf{b},\mathbf{Ax}+\mathbf{B}]=[\mathbf{a},\mathbf{A}]\mathbf{x}+(\mathbf{aB}-\mathbf{Ab}).$$

 In general case, we can understand the covering transform in the following way: We start from $q_0$ in $B_1(0)\subset T_p(M)$. Then exponential map at $p$ maps the point $p_1\in B_1(0)$ to $p\in M$. The derivative maps the tangent vector at $p_1$ to the tangent vector at $p$. Let $\tilde{A}$ be the inverse of the map. Then $F(q_0)=\exp_{p_1}(\tilde{A}q_0)$. In the Ricci flat case, by Lemma \ref{local-geometry}, $g_{ij}$ as well as its $m$-th derivatives are bounded by $C(m,n)\Lambda^2$. So the Christoffel symbols are also bounded as well as their higher derivatives. By the property of ODE, all the parallel transports and the geodesic equations have the same kind of bound as well as their higher derivatives. In particular, the difference between $\tilde{A}$ and the holonomy $A$ of the geodesic loop is bounded by $C(n)\Lambda^2$. The difference between $F(q_0)$ and $p_1+\tilde{A}q_0$ is bounded by $C(n)\Lambda^2$.
 In conclusion, the difference between $F(q_0)$ and $p_1+Aq_0$ is bounded by $C(n)\Lambda^2$ while the difference between their higher derivatives is bound by $C(m,n)\Lambda^2$.

 From now on, we are back to the gravitational instanton $M$ with the point $o$. We will rescale the ball $B_{\mathrm{dist}(o,p)/2}(p)$ to a ball with radius 1 and apply the theory in this section. In particular, the metric on the local covering space is $\delta_{ij}+O'(r^{-\epsilon})$. The difference between the local covering transform with the linear map given by the length, direction, and the holonomy of the geodesic loop is $O'(r^{1-\epsilon})$.

 For short loops, there is a better control given by Buser and Karcher as Proposition 2.3.1 in \cite{BuserKarcher}. They proved that the rotation(i.e. holonomy) part of the Gromov's product of $\gamma_1$ and $\gamma_2$ is given by the calculation in the flat case with error bounded by $Cr^{-2-\epsilon}L(\gamma_1)L(\gamma_2)$, while the error of the translation(i.e. length) part is bounded by $Cr^{-2-\epsilon}L(\gamma_1)L(\gamma_2)(L(\gamma_1)+L(\gamma_2))$.

 \subsection{Control of holonomy of geodesic loops}

 In this section, we will use ODE comparison to study the sliding of geodesic loops and the variation of the induced holonomy.
 First let us recall a well known Jacobi equation
 \[
  J''(t)=(\frac{t}{2})^{-2-\epsilon}J(t).
 \]
 satisfying the following property:
 \begin{proposition} ( c.f. Theorem C of \cite{GreeneWu})Let $J$ be the solution of the Jacobi equation with
 \[
 J(2)=0,\;\;J'(2)=1.
 \]

 Then
 $$1\le J'(t)\nearrow J'(\infty)(:=\lim_{t\rightarrow\infty}J'(t))\le \exp\int_2^{\infty}(t-2)(\frac{t}{2})^{-2-\epsilon}\mathrm{d}t<\infty$$
 and $$t-2\le J(t) \le J'(\infty) (t-2).$$
 \end{proposition}

 Suppose $\gamma$ is a geodesic loop based at $p \in M$, $\alpha$ is an arc-length parameterized curve passing through $p$. Suppose $r =\mathrm{dist}(0,p) = r(p) > 3.\;$ As discussed before, we slide $\gamma$ along $\alpha$ and get a 1-parameter family of geodesic loops $\gamma_t$ based at $\alpha(t)$. Then their length and induced holonomy satisfy the following:

 \begin{proposition}
 Suppose the length and the holonomy of the geodesic loop $\gamma_t$ are $L(t)$ and $\mathrm{hol}(t)$, respectively. Then,
 \[
    |L'(t)| \leq |\mathrm{hol}(t)-\mathrm{Id}|
 \]
 and
 \[
 |\mathrm{hol}(t)-\mathrm{Id}|' \leq L(t) \cdot \displaystyle \max_{x \in \gamma_t} |\mathrm{Rm}|(x).
 \]
 \label{fundamental-equation}
 \end{proposition}

 \begin{proof}
 Let $\gamma(s,t)=\gamma_t(s)$, then $\gamma(0,t)=\gamma(1,t)=\alpha(t)$ and for any fixed $t$, $\gamma(s,t)$ is a geodesic.  So $\partial_s:=\gamma_*(\frac{\partial}{\partial s})$ and $\partial_t:=\gamma_*(\frac{\partial}{\partial t})$ satisfy
 $$\nabla_{\partial_s}\partial_s=0,[\partial_s,\partial_t]=\nabla_{\partial_s}\partial_t-\nabla_{\partial_t}\partial_s=0
 ,L(t)=\int_0^1|\partial_s|\mathrm{d}s$$
 Then,
 $$\begin{array}{lcl} \frac{\mathrm{d}L(t)}{\mathrm{d}t}|_{t=t_0} & = & \int_0^1\frac{<\nabla_{\partial_t}\partial_s,\partial_s>} {<\partial_s,\partial_s>^{1/2}}\mathrm{d}s\\ & = & \frac{1}{L(t_0)}\int_0^1<\nabla_{\partial_s}\partial_t,\partial_s>\mathrm{d}s\\ &= & \frac{1}{L(t_0)}\int_0^1\nabla_{\partial_s}<\partial_t,\partial_s>\mathrm{d}s \\
& = & \frac{<\partial_t,\partial_s>|^{s=1}_{s=0}}{L(t_0)}=<\alpha'(t_0),\frac{(\mathrm{hol-Id})[\partial_s(0,t_0)]}{L(t_0)}>.\end{array} $$
 So $$|L'|\le|\mathrm{hol-Id}|.$$

 Moreover, given any unit length vector $V$ at $\gamma(0,t_0)$, we can parallel transport it along $\alpha(t)=\gamma(0,t)$ and then parallel transport it along $\gamma_t.\;$  Then $\mathrm{hol}(V(0,t))=V(1,t)$.
 So $$
 \begin{array}{lcl} ||\mathrm{hol-Id}|'|&\le & |\nabla_{\partial_t}V(1,t)|  \le  \int_0^1|\nabla_{\partial_s}\nabla_{\partial_t}V|\\
 &
 = & \int_0^1|R(\partial_s,\partial_t)V(s,t)|\le\displaystyle\max_{x\in\gamma_t}|\mathrm{Rm}|L.\end{array} $$

 \end{proof}

 \begin{theorem}
 For any geodesic loop based at $p$ with $r=r(p)=d(p,o)>3$ and length $L\le C_1r$, the holonomy along the loop satisfies
 $$|\mathrm{hol-Id}|\le\frac{J'(r)}{J(r)}L\le C_2 \frac{L}{r}.$$
 Here the constant $$C_1=\frac{1}{2}\inf_{t>2}\frac{t}{J(t)}\inf_{t>3}\frac{J(t)}{t}, C_2=\sup_{t>3}J'(t)\sup_{t>3}\frac{t}{J(t)}.$$
 \label{hol-control}
 \end{theorem}

 \begin{proof}
 If we choose $\alpha(t)$ so that $\partial_t=\frac{\mathrm{hol-Id}}{|\mathrm{hol-Id}|}\frac{\partial_s}{L(t_0)}$, we can get $L'(t)=|\mathrm{hol-Id}|$. It is some kind of gradient flow.
 The other fundamental equation is that $|\mathrm{hol-Id}|'$ is bounded by the product of $L$ and the maximal Riemannian curvature along the geodesic loop.

 Given $p$ whose distance to origin $r=r(p)=d(p,o)>3$ and any geodesic loop based at $p$ with length smaller than $C_1r<\frac{r}{2}$, if $|\mathrm{hol-Id}|>\frac{J'(r)}{J(r)}L$, we can slide the curve back along the gradient flow. In other words, we start from $\alpha(r)=p$ and get a curve $\alpha:[t_1,r]\rightarrow M$ as well as the corresponding $\gamma_t$. Let $t_1$ be the biggest $t_1$ such that one of the following happens: (1) $L(t_1)=t_1/2$; (2) $L'(t_1)=|\mathrm{hol-Id}|=0$ or $L(t_1)=0$; (3) $t_1=2$. Then when $t\in (t_1,r)$, we have $0<L(t)<t/2$  and $t>t_1\ge 2$. So the distance to the origin is at least $t-L(t)>t/2$. The curvature is bounded by $(t/2)^{-2-\epsilon}$ and the conjugate radius is at least $\pi(\frac{t}{2})^{1+\epsilon/2}>\frac{t}{2}>L(t)$. So the geodesic loop can exist without going out of the conjugate radius.
 Combining two fundamental equations together,we have
 $$L''(t)\le L(t)\mathrm{max|Rm|}\le L(t)(t-L(t))^{-2-\epsilon}< L(t)(\frac{t}{2})^{-2-\epsilon},\forall t\in (t_1,r).$$
 Therefore $(L'J-J'L)'=L''J-J''L< 0$. By our hypothesis $L'(r)>\frac{J'(r)}{J(r)}L(r)$. So $L'(t)J(t)-J'(t)L(t)>0\Rightarrow(\frac{L(t)}{J(t)})'>0\Rightarrow \frac{L(t)}{J(t)}<\frac{L(r)}{J(r)},\forall t\in [t_1,r)$.
 So $L(t_1)<\frac{L(r)}{J(r)}J(t_1)\le C_1\frac{r}{J(r)}\frac{J(t_1)}{t_1}t_1\le\frac{t_1}{2}$ and $L'(t_1)J(t_1)>J'(t_1)L(t_1)\ge 0$. In other words, $t_1=2$. But then $L(2)<\frac{L(r)}{J(r)}J(2)=0$. It is a contradiction.
 \end{proof}

 For any fixed geodesic ray $\alpha$ starting from $o$, any number $r>3$ and any geodesic loop $\gamma$ based at $p=\alpha(r)$ with length $L\le C_1r$, when we slide it along the ray towards infinity, it will always exist i.e. stay within the conjugate radius. This follows from the following rough estimate:
 \begin{corollary}
  The length $L(t)$ of the geodesic loop based at $\alpha(t)$ is smaller than $t/2$ for all $t\ge r$.
 \end{corollary}

 \begin{proof}
 By Proposition \ref{fundamental-equation} and Theorem \ref{hol-control}, we know that
 $L'(t)\le \frac{J'(t)}{J(t)}L(t).$
 So $$(\ln L)'\le (\ln J)'\Rightarrow L(t)\le \frac{L(r)}{J(r)}J(t)\le \frac{t}{2},\forall t\ge r>3.$$
 \end{proof}
 We will derive a better estimate and use it to prove the first main theorem.
\subsection{Classification of tangent cone at infinity}
 To under how the length of geodesic loops varies, we first need to understand the structure at infinity. Our assumption of the decay of the curvature means that we are at a manifold with asymptotically nonnegative curvature.  The end of such a manifold is well studied and goes back to Kasue \cite{Kasue}. Here, a complete connected noncompact Riemannian manifold $M$ with a base point $o$ is called asymptotically nonnegative curved if there exists a monotone nonincreasing function $k:[0,\infty)\rightarrow[0,\infty)$ such that the integral $\int_0^{\infty}tk(t)\mathrm{d}t$ is finite and the sectional curvature of $M$ at any point $p$ is bounded from below by $-k(\mathrm{dist}(o,p))$. Of course, the gravitational instanton $M$ satisfies this condition.

 \begin{theorem}
 (\cite{Kasue} \cite{Drees} \cite{MashikoNaganoOtsuka})
Let $M$ be a manifold with asymptotically nonnegative curvature. Two rays $\sigma$ and $\gamma$ starting from $o$ are called equivalent if $\lim_{t\rightarrow \infty}\mathrm{dist}(\sigma(t),\gamma(t))/t=0$. Denote the set of equivalent classes of geodesic rays starting from $o$ by $S(\infty)$. Then there exists a metric $\delta_{\infty}$ on $S(\infty)$ such that $(S(\infty),\delta_{\infty})$ forms a compact inner metric space, in other words, length space. Consider the cone $C(S(\infty))$ over $S(\infty)$ with the natural distance $$\Delta_{\infty}((t,p),(t',p'))=\sqrt{t^2+t'^2-2tt'\cos(\min\{\pi,\delta_{\infty}(p,p')\})}.$$ Fix the representative $\sigma$ from each equivalent class $[\sigma]$. Define the map $\Phi_t: \{r\in[a,b]\}\cap C(S(\infty))\rightarrow \{r\in[at,bt]\}\cap M$ by $\Phi_t(r,[\sigma])=\sigma(rt)$ for any fixed $0<a<b<\infty$ and any $t>0$, . Then the Gromov-Hausdorff distance between $(\{r\in[a,b]\}\cap C(S(\infty)),\Delta_\infty)$ and $(\{r\in[at,bt]\}\cap M,\mathrm{dist}/t)$ using the map $\Phi_t$ converges to 0 when $t$ goes to infinity. In other words, the tangent cone at infinity is unique and must be a metric cone $C(S(\infty))$.
\label{uniqueness-of-tangent-cone}
 \end{theorem}
 
 \begin{remark}
 Drees \cite{Drees} pointed out a gap in \cite{Kasue}. It was corrected by Mashiko, Nagano and Otsuka \cite{MashikoNaganoOtsuka}.
 \end{remark}

The following additional thing is true for gravitational instantons:

 \begin{theorem}
 (ALH-splitting) If the $S(\infty)$ of a gravitational instanton $M$ has more than one connected components, $M$ must be isometric to the product of $\mathbb{R}$ and a flat 3-torus.
 \label{one-end}
 \end{theorem}

 \begin{proof}
If $S(\infty)$ has more than one connected components, we can find a large enough ball $B_R$ and two sequences $p_i$, $q_i$ such that $\mathrm{dist}(o, p_i)\rightarrow\infty$, $\mathrm{dist}(o, q_i)\rightarrow\infty$, and any minimal geodesics connecting $p_i$ and $q_i$ must pass through $B_R$ for any $i$ large enough. By compactness of $B_R$, the minimal geodesics converge to a line. Notice that $M$ is Ricci-flat, so the splitting theorem \cite{CheegerGromoll} implies that $M$ must be isometric to the product of $\mathbb{R}$ and a 3-manifold. The 3-manifold is also Ricci-flat and therefore flat. Now any geodesic loop in this 3-manifold must have the trivial holonomy by Theorem \ref{hol-control}. So it must be a 3-torus.
 \end{proof}

From now on, we assume that $S(\infty)$ has only one component.

As a corollary, the following is true:

 \begin{corollary}
 Fix a ray $\gamma$ starting from $o$.There is a constant $C_3$ such that for any point $p$ in the large enough sphere $S_{r(p)}$, there is a curve within $B_{1.1r(p)}\setminus B_{0.9r(p)}$ connecting $p$ and $\gamma(r(p))$ with length bounded by $C_3r(p)$.
 \label{connecting-points-with-ray}
 \end{corollary}

There is more information about the tangent cone at infinity of the gravitational instanton $M$.

 \begin{theorem}
 The tangent cone at infinity $C(S(\infty))$ of the gravitational instanton $M$ must be a flat manifold with only possible singularity at origin.
 \label{local-tangent-cone}
 \end{theorem}

 \begin{proof}
 Pick $p\in C(S(\infty))-\{o\}$, we may find $p_i\in M$ such that $p_i\rightarrow p$ in Gromov-Hausdorff sense. Pick some small enough number $\kappa$. For $i$ large enough the ball $(B_{\kappa r_i}(p_i),r_i^{-2}g)$ is $B_{\kappa}/G_i$, where $B_{\kappa}$ is the ball in the Euclidean space with metric pulled back by exponential map, and $G_i$ is the group of local covering transforms.
 By Fukaya's result in \cite{Fukaya}, $G_i$ converge to some Lie group $G$ and $B_{\kappa}/G_i$ converge to $B_{\kappa}/G$. So $G$ is a subgroup of $\mathbb{R}^4 \rtimes \mathrm{SU}(2)\le\mathrm{Iso}(\mathbb{R}^4)$. The action of $G$ on $B_{\kappa}$ corresponds to the action of $G_i$ on $B_{\kappa r_i}(p_i)$. So if an element $g\in G-\{\mathrm{Id}\}$ has a fixed point in $B_{\kappa}$, the geodesic loops in $B_{\kappa r_i}(p_i)$ corresponding to the sequence $g_i\in G_i$ converging to $g$ would have large $|\mathrm{hol-Id}|$ compared to their lengths by the relationship between geodesic loops and covering transforms. This contradicts Theorem \ref{hol-control}. So the action of $G$ is free. Therefore it is enough to look at the Lie algebra $\mathfrak{g}$ i.e. the infinitesimal part of $G$ to determine the local geometry. We have the following cases:

 (0) $\mathrm{dim}G=0$. We get $\mathbb{R}^4$ locally.

 (1) $\mathrm{dim}G=1$. Then $\mathfrak{g}$ is generated by $\mathbf{x}\rightarrow \mathbf{a}\mathbf{x}+\mathbf{b}$, where $\mathbf{a}\in\mathfrak{su}(2)$.
 Notice that $\mathrm{SU}(2)$ can be naturally identified with the unit sphere of quaternions. Then $\mathfrak{su}(2)$ would be the space of pure imaginary quaternions. So the Lie bracket is
 exactly twice of the cross product in $\mathbb{R}^3$.

 $\mathbf{a}$ must be $\mathbf{O}$ or invertible by the property of quaternions.  When $\mathbf{a}=\mathbf{O}$, $G$ consists of pure translations, we get $\mathbb{R}^3$.

 Otherwise, $\mathbf{a}\mathbf{x}+\mathbf{b}=\mathbf{a}(\mathbf{x}+\mathbf{a}^{-1}\mathbf{b})$. The fixed point $-\mathbf{a}^{-1}\mathbf{b}$ must be outside $B_{\kappa}$. $G$ is generated by $\mathbf{x}\rightarrow \left( {\begin{array}{*{20}c}
    e^{i\theta} & 0  \\
    0 & e^{-i\theta}   \\
 \end{array}} \right) (\mathbf{x}+\mathbf{a}^{-1}\mathbf{b}) -\mathbf{a}^{-1}\mathbf{b}$.
 If we take the 1-1 correspondence $\mathbf{x}\rightarrow\mathbf{x}+\mathbf{a}^{-1}\mathbf{b}=(x+iy,z+iw)\rightarrow(x+iy,z-iw)$,then $G$ becomes $\left( {\begin{array}{*{20}c}
    e^{i\theta} & 0  \\
    0 & e^{i\theta}   \\
 \end{array}} \right)$. So it is cone over $\mathbb{S}^3/\mathbb{S}^1$, where $\mathbb{S}^3/\mathbb{S}^1$ is the Hopf fiberation. So $B_{\kappa}$ is a local piece of the cone over $\mathbb{S}^2$, in other words, $\mathbb{R}^3$, too.

 (2) $\mathrm{dim}G=2$. Any 2-dimensional Lie algebra has a basis $e_1,e_2$ satisfying $[e_1,e_2]=ce_1$. For $\mathfrak{g}$, $e_1(\mathbf{x})=\mathbf{a}\mathbf{x}+\mathbf{b},e_2(\mathbf{x})=\mathbf{Ax}+\mathbf{B}$ must satisfy $$[\mathbf{a},\mathbf{A}]\mathbf{x}+(\mathbf{aB}-\mathbf{Ab})
 =[\mathbf{ax}+\mathbf{b},\mathbf{Ax}+\mathbf{B}]=ce_1=c(\mathbf{ax}+\mathbf{b}).$$
 Here $\mathbf{A},\mathbf{a}\in \mathfrak{su}(2)$.
 If $\mathbf{a}=\mathbf{O}$, $\mathbf{Ab}=-c\mathbf{b}$. So $\mathbf{A}=\mathbf{O}$. $G$ consists of pure translations, we get $\mathbb{R}^2$.
 If $\mathbf{a}\not=\mathbf{O}$, then since $[\mathbf{a},\mathbf{A}]=c\mathbf{a}$, we must have $\mathbf{a}=\mathbf{A}$, and $c=0$.
 So $\mathbf{aB}=\mathbf{Ab}=\mathbf{ab}\Rightarrow \mathbf{B}=\mathbf{b}$, contradiction.

 (3) $\mathrm{dim}G=3$. We get $\mathbb{R}^1$.
 \end{proof}

 \begin{theorem}
 The tangent cone at infinity $C(S(\infty))$ must be the following:

 (ALE) $\mathbb{R}^4/\Gamma$, where $\Gamma$ is a discrete subgroup of $O(4)$ acting freely on $\mathbb{S}^3$

 (ALF-$A_k$) $\mathbb{R}^3$

 (ALF-$D_k$) $\mathbb{R}^3/\mathbb{Z}_2$=cone over $\mathbb{RP}^2$

 (ALG) flat cone with angle $\in(0,2\pi]$

 (ALH-non-splitting) $\mathbb{R}_+$
 \label{classify-tangent-cone}
 \end{theorem}

 \begin{proof}
 By Theorem \ref{uniqueness-of-tangent-cone}, the tangent cone at infinity is unique and must be a metric cone $C(S(\infty))$.
 By Theorem \ref{one-end} and Theorem \ref{local-tangent-cone}$, S(\infty)$ is a connected manifold since we have assumed that $M$ is not ALH-splitting.

 (ALH-nonsplitting) If $S(\infty)$ is 0-dimensional, $C(S(\infty))$ must be $\mathbb{R}_+$.

 (ALG) If $S(\infty)$ is 1-dimensional, $C(S(\infty))$ is a flat cone. If the cone angle is bigger than $2\pi$,
 it contains a line, so there is a contradiction from the almost splitting theorem. (c.f. Theorem 6.64 of \cite{CheegerColding})

 (ALF) If $S(\infty)$ is 2-dimensional, $S(\infty)$ must be a 2-manifold with constant positive curvature 1. So its universal cover is the space form $\mathbb{S}^2$. So $S(\infty)=\mathbb{S}^2/\Gamma$, where the group of covering transforms $\Gamma$ is a subgroup of $\mathrm{Iso}(\mathbb{S}^2)=\mathrm{O}(3)$ acting freely. Now pick any element $A$ in $\Gamma$, $A^2\in \mathrm{SO}(3)$. However, any element in $\mathrm{SO}(3)$ has a fixed point, so $A^2=\mathrm{Id}$. So $A=\pm\mathrm{Id}$. Therefore $S(\infty)$=$\mathbb{S}^2$(the $A_k$ case) or $\mathbb{RP}^2$(the $D_k$ case).

 (ALE) If $S(\infty)$ is 3-dimensional, $S(\infty)$ must has constant sectional curvature, too. Its universal cover is the space form, too.
 So $C(S(\infty))=\mathbb{R}^4/\Gamma$, where $\Gamma$ is a discrete subgroup of $O(4)$ acting freely on $\mathbb{S}^3$
 \end{proof}

 From now on, we use the terminology ALE,ALF,ALG,and ALH to distinguish different type of the (unique) tangent cone at infinity. Those terminologies make sense after we prove more properties.

 \begin{theorem}
 In the ALE case, $M$ has maximal volume growth rate and it is in Kronheimer's list.
 \label{ALE}
 \end{theorem}
 \begin{proof}
By Colding’s volume convergence theorem \cite{Colding}, $M^4$ has maximal volume growth rate. Moreover, the faster than quadratic curvature decay condition ensures that $\int_M|\mathrm{Rm}|^2<\infty$. So by Bando-Kasue-Nakajima’s work \cite{BandoKasueNakajima},
$M$ is ALE of order 4. So Kronheimer’s works in \cite{Kronheimer1} and \cite{Kronheimer2} apply.
 \end{proof}

\subsection{Decomposing geodesic loops into basis}
 Before proceeding, we need a theorem about Lie groups. For any Lie group $H$, the exponential map $\exp$ from a small ball $B_{\kappa}=B_{\kappa}(o)$ in its Lie algebra $\mathfrak{h}$ to $H$ is a bijection. We call the inverse of $\exp$ to be $\log$. If there is no ambiguity, the length of $g\in H$ will mean $|\log g|$.

 \begin{theorem}
 (Theorem 4.5 of \cite{BuserKarcher}) Suppose that $H$ is a Lie group, $G_i$ are discrete subgroups of $H$ converging to a $k$-dimensional Lie subgroup $G$ of $H$. Then for $i$ large enough and $\kappa$ small enough, there exist $k$ elements $g_{i,j}$(j=1,2,...,k) such that $|\log g_{i,j}|$ converge to 0 as $i$ goes to infinity and all element in $B_{\kappa}(\mathrm{Id})\cap G_i$ is generated by $g_{i,j}$. What is more, for any fixed large enough $i$, the angle between $\log g_{i,j}$ are bounded from below by a small positive number independent of $i$. In addition, the commutator of $g_{i,a}$ and $g_{i,b}$ is generated by $g_{i,c}$ for $c=1,2,...,\min\{a,b\}-1$. In particular, $g_{i,1}$ commutes with others.
 \label{generator}
 \end{theorem}

 \begin{remark}
 According to Theorem 4.5 of \cite{BuserKarcher}, it is enough to assume that $G_i$ have a local group structure near identity rather than being a group. Actually the theorem is even true if the product of $a,b\in G_i$ contains an error controlled by $C_i|a||b|$, where $C_i$ converge to 0 as $i$ goes to infinity. In particular, the local groups $G_i$ in Theorem \ref{local-tangent-cone} satisfy the Theorem. For those local groups, since the rotation (i.e holonomy) part is bounded by the translation (i.e. length) part by Theorem \ref{hol-control}, the length of the geodesic loop is equivalent to the length in the above theorem.
 \end{remark}

 Now we are ready to go back to study the length of short geodesic loops.  In the rest of this section, we fix a geodesic ray $\alpha$ from $o$ to infinity and start doing analysis about geodesic loops based on the ray. This ray corresponds to the point $(1,[\alpha])$ in the tangent cone at infinity $C(S(\infty))$.

 \begin{theorem}
 In the ALF-$A_k$ or ALF-$D_k$ cases, there is a geodesic loop $\gamma_1$ such that when we slide it along the fixed ray to get $\gamma_{r,1}$ based at $\alpha(r)$, its length
 $L(r):=L(\gamma_{r,1})=L_{\infty}+O(r^{-\epsilon})$ and its holonomy satisfies $|\mathrm{hol-Id}|=O(r^{-1-\epsilon})$.
 What is more, any loop based at $\alpha(r)$ with length smaller than $\kappa r$ is generated by $\gamma_{r,1}$ in the sense of Gromov.
 \label{3d-tangent-cone}
 \end{theorem}

 \begin{proof}
 In this case, $(B_{\kappa r}(\alpha(r)),r^{-2}g)$ converge to $B_{\kappa}((1,[\alpha]))\subset C(S(\infty))$ by Theorem \ref{local-tangent-cone}. We may make $\kappa$ even smaller to apply Theorem \ref{generator}. We get $\gamma_{r,1}$ corresponding to $g_{r,1}$ in Theorem \ref{generator}. Then any loop based at $\alpha(r)$ with length smaller than $\kappa r$ is generated by $\gamma_{r,1}$ in the sense of Gromov.
 There is an ambiguity to choose $\gamma_{r,1}$. The same loop with reverse direction would play the same role.
 However, we can choose them consistently so that they are the sliding of each other along the ray.
 By Theorem \ref{generator},
 $$\lim_{r\rightarrow \infty} {\frac{L(r)}{r}} = 0.$$
 So the holonomy along the loop converges to identity by Theorem \ref{hol-control}.   It follows that
 $$
 \begin{array} {lcl} |\mathrm{hol-Id}|(r) & = &   |\mathrm{hol-Id}| (\infty) -\displaystyle \int_r^\infty\; |\mathrm{hol-Id}|' \mathrm{d}t\\
 & \le & 0 +  \displaystyle \int_r^\infty\; CLt^{-2-\epsilon} \mathrm{d}t  \\
 & \le & O(r^{-\epsilon}).\end{array}$$
 by the equation that $||\mathrm{hol-Id}|'| < CLr^{-2-\epsilon}$. Plug this back to the equation
 \[
  |L'| \leq  |\mathrm{hol-Id}|,
  \]
 we obtain
 $$
 \begin{array}{lcl} L(r) & = & L(r_0) + \displaystyle \int_{r_0}^{r} L'(t) \mathrm{d}t \leq  L(r_0) + \displaystyle \int_{r_0}^{r}  |\mathrm{hol-Id}| \mathrm{d}t\\
 & \le  & L(r_0) +   \displaystyle \int_{r_0}^{r} C t^{-\epsilon} d\,t  = L(r_0) + C (r^{1-\epsilon}-r_0^{1-\epsilon}).
 \end{array} $$
 In turn $|\mathrm{hol-Id}|\le O(r^{-2\epsilon})$,
 $L\le O(r^{1-2\epsilon}).\cdots$ Through finite steps of iterations, we have
 $$L=L_\infty+O(r^{-\epsilon})\qquad {\rm and}\qquad |\mathrm{hol-Id}|\le O(r^{-1-\epsilon}). $$

 \noindent {\bf Claim}:  The limit length $L_\infty = \displaystyle \lim_{t\rightarrow \infty} L_t  > 0$. Otherwise, since $L=O(r^{-\epsilon})$, after the integration from infinity to $r$, we can
 easily obtain
 $$|\mathrm{hol-Id}|\le O(r^{-1-2\epsilon})$$
 After a finite number of iterations,  we have $$L=O(r^{-1-\epsilon}) \qquad {\rm and} \qquad |\mathrm{hol-Id}|\le O(r^{-2-2\epsilon}).$$
 Now let $$f(r)=\sum_{k=0}^{\infty}\frac{2^{(2+\epsilon)k}r^{-k\epsilon}}
 {\epsilon(\epsilon+1)2\epsilon(2\epsilon+1)...k\epsilon(k\epsilon+1)},$$
 Then $$f''(r)=(\frac{r}{2})^{-2-\epsilon}f(r), \displaystyle f(r)=1+O(r^{-\epsilon}), \displaystyle f'(r)=O(r^{-1-\epsilon}).$$
 So for all $R$ large enough, we have $$L(R)<R^{-1}f(R)$$ and
 $$|L'(R)|<R^{-1}|f'(R)|.$$
 By ODE comparison, we have $L(r)<R^{-1}f(r)$.  Let $R$ go to infinity, $L(r)=0$, this is a contradiction. So $L_\infty>0$.
 \end{proof}

 \begin{theorem}
 In the ALG case, there are commutative geodesic loops $\gamma_1,\gamma_2$ such that when we slide them along the fixed ray $\alpha$ to get $\gamma_{r,1},\gamma_{r,2}$ based at $\alpha(r)$, their length
 $L_j(r):=L(\gamma_{r,j})=L_{\infty,j}+O(r^{-\epsilon})$ and their holonomy satisfy $|\mathrm{hol-Id}|=O(r^{-1-\epsilon})$.
 What is more,any loop based at $\alpha(r)$ with length smaller than $\kappa r$ is generated by $\gamma_{r,1}$ and $\gamma_{r,2}$ in the sense of Gromov.
 \label{2d-tangent-cone}
 \end{theorem}

 \begin{proof}
 We proceed as in the proof of Theorem \ref{3d-tangent-cone}. We get two loops $\gamma_{r,1}$ and $\gamma_{r,2}$ based at $\alpha(r)$. In this case, the ambiguity is as large as $\mathrm{GL}(2,\mathbb{Z})$. In other words, $\gamma_{r,1}$ and $\gamma_{r,2}$ may jump to $\gamma_{r,1}^{100}\gamma_{r,2}^{99}$ and $\gamma_{r,1}^{101}\gamma_{r,2}^{100}$ respectively after the sliding. Actually $\mathrm{GL}(2,\mathbb{Z})$ is a noncompact group, so we can not estimate the length of the geodesic loops obtained by sliding directly. However, we can still get the same conclusion from the fact that $\gamma_{r,1}$ and $\gamma_{r,2}$ commute and that they form a detectable angle.

 Suppose the manifold is flat, then the covering transforms corresponding to $\gamma_{r,1}$ and $\gamma_{r,2}$ are linear maps
 $T_1(\mathbf{x})=\mathbf{ax}+\mathbf{b}$, $T_2(\mathbf{x})=\mathbf{Ax}+\mathbf{B}$, where
 $\mathbf{a},\mathbf{A}\in\mathrm{SU}(2), \mathbf{b},\mathbf{B}\in\mathbb{C}^2$.
 So (Note that by the construction $|\mathbf{b}|<C|\mathbf{B}|$) $$\mathbf{x}=T_1^{-1}T_2^{-1}T_1T_2(\mathbf{x})=\mathbf{a}^{-1}\mathbf{A}^{-1}\mathbf{aAx}+
 \mathbf{a}^{-1}\mathbf{A}^{-1}((\mathbf{a}-\mathbf{Id})\mathbf{B}+(\mathbf{Id}-\mathbf{A})\mathbf{b}).$$
 On the manifold, we need to count the error caused by curvature.
 So actually $$|(\mathbf{a}-\mathbf{Id})\mathbf{B}-(\mathbf{A}-\mathbf{Id})\mathbf{b}|\le Cr^{-2-\epsilon}|\mathbf{b}||\mathbf{B}|^2,
 |\mathbf{a}^{-1}\mathbf{A}^{-1}\mathbf{aA}-\mathbf{Id}|<Cr^{-2-\epsilon}|\mathbf{b}||\mathbf{B}|.$$

 Now if $|\mathbf{a}-\mathbf{Id}|>r^{-1-\epsilon/3}|\mathbf{b}|$,  then $|\mathbf{A}-\mathbf{Id}||\mathbf{b}|\ge|\mathbf{a}-\mathbf{Id}||\mathbf{B}|-Cr^{-2-\epsilon}|\mathbf{b}||\mathbf{B}|^2.\;$ It follows that  $|\mathbf{A}-\mathbf{Id}|> c \cdot r^{-1-\epsilon/3}|\mathbf{B}|$ for some constant $c.$ Thus, if $r$ is large enough,  the two vectors
 $(\mathbf{A}-\mathbf{Id})\mathbf{b}$ and $(\mathbf{a}-\mathbf{Id}) \mathbf{B}$  have almost the same angle since their difference has much smaller length.  Note that both $\mathbf{A}$ and $\mathbf{a}$ are very close to identity, so $\mathbf{A}-\mathbf{Id}$ and $\mathbf{a}-\mathbf{Id}$ are almost $\log(\mathbf{A})$ and $\log(\mathbf{a})$ respectively.  So Theorem \ref{generator} is reduced to that $(\mathbf{a}-\mathbf{Id},\mathbf{b})$ form a detectable angle with $(\mathbf{A}-\mathbf{Id},\mathbf{B})$. Therefore, $\mathbf{A}-\mathbf{Id}$ and $\mathbf{a}-\mathbf{Id}$ also form a detectable angle because $(\mathbf{A}-\mathbf{Id})\mathbf{b}$ has almost
 the same angle with $(\mathbf{a}-\mathbf{Id})\mathbf{B}$.

 Since the Lie algebra in $\mathfrak{su}(2)$ is simply the cross product and all the matrices are very close to identity
 $$|\mathbf{a}-\mathbf{Id}||\mathbf{A}-\mathbf{Id}|<C|\mathbf{a}^{-1}\mathbf{A}^{-1}\mathbf{a}\mathbf{A}-\mathbf{Id}|
 <Cr^{-2-\epsilon}|\mathbf{b}||\mathbf{B}|.$$
 This is a contradiction.   So $|\mathbf{a}-\mathbf{Id}|\le r^{-1-\epsilon/3}|\mathbf{b}|$. Similarly $|\mathbf{A}-\mathbf{Id}|\le r^{-1-\epsilon/3}|\mathbf{B}|$.

 We have proved that for $\gamma_{r,1}$ and $\gamma_{r,2}$, $|\mathrm{hol-Id}|\le r^{-1-\epsilon/3}L$. For any loop with length smaller than $\kappa r$, we have $|\mathrm{hol-Id}|\le Cr^{-\epsilon/3}$.  When we slide $\gamma_{r,j}$ along the fixed ray towards infinity, the holonomy of the limiting loops must be trivial. The proof in Theorem \ref{3d-tangent-cone} then implies our conclusion. Note that the ambiguity of choosing $\gamma_{r,j}$ now can be removed by requiring that they are the sliding of loops along $\alpha$.
 \end{proof}

 \begin{theorem}
 In the ALH-non-splitting case, there are commutative geodesic loops $\gamma_1,\gamma_2,\gamma_3$ such that when we slide them along the fixed ray $\alpha$ to get $\gamma_{r,1}$, $\gamma_{r,2}$, $\gamma_{r,3}$ based at $\alpha(r)$, their length
 $L_j(r):=L(\gamma_{r,j})=L_{\infty,j}+O(r^{-\epsilon})$ and their holonomy satisfy $|\mathrm{hol-Id}|=O(r^{-1-\epsilon})$.
 What is more, any loop based at $\alpha(r)$ with length smaller than $\kappa r$ is generated by $\gamma_{r,1}$,$\gamma_{r,2}$ and $\gamma_{r,3}$ in the sense of Gromov.
 \label{1d-tangent-cone}
 \end{theorem}

 \begin{proof}
 We can proceed exactly in the same way as Theorem \ref{2d-tangent-cone}. The only thing we need to prove is that $\gamma_{r,2}$ commutes with $\gamma_{r,3}$. It follows from the fact that the length of the commutator converge to 0 since the curvature and therefore the errors converge to 0 as $r$ goes to 0.
 \end{proof}

 \subsection{From geodesic loops to Riemannian fiberation}

 In \cite{CheegerFukayaGromov},  Cheeger, Fukaya and Gromov first introduced the N-structure i.e. nilpotent group fiberations of different dimensions patched together consistently. (Torus is the simplest nilpotent group.)
 In \cite{MinerbeAsymptotic}, Minerbe followed their method and improved the result for circle fiberations under a strong volume growth condition in ALF case.
 In their papers they all view $\mathbb{R}^{4-k}\times\mathbb{T}^k$ as the Gromov-Hausdorff approximation of $\mathbb{R}^{4-k}$.
 In this subsection, we also include the $\mathbb{T}^k$ factor in the analysis.  Therefore, we are able to obtain a better estimate without any volume assumptions.

 In the last subsection, we get geodesic loops $\gamma_{p,i}$ along a ray. They can be represented by $s\in[0,1]\rightarrow \exp_p(sv_i(p))$ for some vectors $v_i(p)$ in the tangent space of the base point $p$. When $p$ goes to infinity, the vectors $v_i(p)$ converge to some limits $v_i\in \mathbb{R}^4$. Actually, the difference between $v_i(p)$ and $v_i$ is bounded by $O(r^{-\epsilon})$. Define the lattice $\Lambda$ by $\Lambda=\oplus_{i=1}^{k}\mathbb{Z}v_i$ and the torus $\mathbb{T}^k=(\oplus_{i=1}^{k}\mathbb{R}v_i)/\Lambda$ with the induced metric. From the estimates in the last subsection and the estimates in the last paragraph of Section 2.1 (c.f. Proposition 2.3.1 of \cite{BuserKarcher}), it is easy to see that for $\sum_{i=1}^{k}a_iv_i\in\Lambda\cap B_{\kappa r(p)}$, the translation part of the Gromov product $\prod_{i=1}^{k}\gamma_{p,i}^{a_i}$ is $\sum_{i=1}^{k}a_iv_i$ with error bounded by $O(r^{1-\epsilon})$ while the holonomy is bounded by $O(r^{-\epsilon})$. So the lattice $\Lambda$ almost represent the geodesic loops whose length is smaller than $\kappa r(p)$.

By Proposition \ref{fundamental-equation}, Corollary \ref{connecting-points-with-ray}, the estimates in Theorem \ref{3d-tangent-cone}, Theorem \ref{2d-tangent-cone} or Theorem \ref{1d-tangent-cone}, we can slide the geodesic loops $\gamma_{p,i}$ along a path within $B_{1.1r}(o)\setminus B_{0.9r}(o)$ to get geodesic loops $\gamma_{p,i}$ over the whole manifold $M$ except a compact set $K$. It satisfies all the above properties. The choice of path is not unique, so after sliding along different paths, $\gamma_{p,i}$ may be different. However, all the differences come from a change of basis in $\Lambda$. Locally, we can assume that $\gamma_{p,i}$ are well defined.

  \begin{theorem}
 We can find a diffeomorphism from $B_{\kappa r}(p)$ to $B_{\kappa r}(0)\times \mathbb{T}^k\subset\mathbb{R}^{4-k}\times \mathbb{T}^{k}$,
 such that $g$ = the pull back of the flat metric + $O'(r^{-\epsilon}).$
 \label{local-fiberation}
 \end{theorem}

 \begin{proof}
 First of all we look at the map $\exp:\mathrm{T}_p\rightarrow M$. Any $q\in B_{\kappa r}(p)$ has lots of preimages. Choose one preimage $q_0$, then all the other preimages are $\prod_{i=1}^{k} F_i^{a_i}(q_0)$, where $F_i$ are the covering transforms corresponding to $\gamma_i$, and $a_i$ are integers. We know that $\prod_{i=1}^{k} F_i^{a_i}(q_0)$ is actually $q_0+\sum_{i=1}^{k}a_iv_i$ with error in $O'(r^{1-\epsilon})$.
 Define
 $$f(q)=\pi_{\mathbb{T}^k}\frac{\sum \chi(\frac{10|\prod_{i=1}^{k} F_i^{a_i}(q_0)|}{\kappa r(p)})(\prod_{i=1}^{k} F_i^{a_i}(q_0)-\sum_{i=1}^{k}a_iv_i)}{\sum \chi(\frac{10|\prod_{i=1}^{k} F_i^{a_i}(q_0)|}{\kappa r(p)})}\in \mathbb{R}^{4-k}\times \mathbb{T}^{k},$$
 then it is independent of the choice of $q_0$. Roughly speaking, $f(q)$ is the weighted average of the projections of all the preimages of $q_0$ to $\mathbb{R}^{4-k}\times \mathbb{T}^{k}$. It is easy to prove that using $f$, the metric $g$= the pull back of the flat metric + $O'(r^{-\epsilon}).$
 \end{proof}

 \begin{lemma}
 We can find good covers $\{B_{\frac{1}{2}\kappa r(p_i)}(p_i)\}_{i\in I}$ such that $I$ can be divided into $I=I_1\cup...\cup I_N$, and if $i,j\in I_l$, $l=1,2,...,N$, $B_{\kappa r(p_i)}(p_i)\cap B_{\kappa r(p_j)}(p_j)=\emptyset$.
 \label{good-cover}
 \end{lemma}

 \begin{proof}
 This kind of theorem was first proved in \cite{CheegerFukayaGromov}. In our situation we can choose maximal $\kappa2^{l-1}$ nets in $B(2^{l+1})-B(2^l)$. Then volume comparison implies the property.
 \end{proof}

 \begin{theorem}
 Outside a compact set $K$, there is a global fiberation and a $\mathbb{T}^k$ invariant metric $\tilde g=g+O'(r^{-\epsilon})$ whose curvature belongs to $O'(r^{-2-\epsilon})$
 \label{fiberation}
 \end{theorem}

 \begin{proof}
 By Lemma \ref{good-cover}, we can first modify $i\in I_1$ and $j\in I_2$ so that they are compatible. Then modify $i,j\in I_1,I_2$ and $l\in I_3$ to make sure that they are compatible. After $N$ times, we are done.
 So we start from a map $$f_{ij}:B_{\kappa r_i}(p_i)\times \mathbb{T}^k\rightarrow B_{\kappa r_j}(p_j)\times \mathbb{T}^k.$$
 $$f_{ij}(q,\theta)=(f_{ij}^1(q,\theta),f_{ij}^2(q,\theta))=f_j\circ f_i^{-1}(q,\theta).$$
 Average it and get $\tilde f^1_{ij}:B_{\kappa r_i}(p_i)\rightarrow B_{\kappa r_j}(p_j)$ by
 $$\tilde f_{ij}(q)=\frac{1}{\mathrm{Vol}(\mathbb{T}^k)}\int_{\mathbb{T}^k}f^1_{ij}(q,\theta)\mathrm{d}\theta.$$
 From the higher derivative control, we know that the distance from origin to
 $f^2_{ij}(q,\theta)-f^2_{ij}(q,0)-\theta\in\mathbb{T}^k$ is bounded by $O(r^{-\epsilon})$. (Here we view $\mathbb{T}^k$ as an abelian group.) For $r$ large enough, we can lift it to $\mathbb{R}^k$ while keeping it bounded by $O(r^{-\epsilon})$. Fix $q$ and average it with respect to $\theta$, then project it back to $\mathbb{T}^k$. We get a map $\tilde f^2_{ij}:B_{\kappa r_i}(p_i)\rightarrow \mathbb{T}^k$. Define $\tilde f_{ij}:B_{\kappa r_i}(p_i)\times \mathbb{T}^k\rightarrow B_{\kappa r_j}(p_j)\times \mathbb{T}^k$ by
 $$\tilde f_{ij}(q,\theta)=(\tilde f^1_{ij}(q),\theta+f^2_{ij}(q,0)+\tilde f^2_{ij}(q)).$$
 It is easy to see that $|\nabla^m \tilde f_{ij}|=O(r^{1-m-\epsilon})$.
 We may glue the common part using $\tilde f_{ij}$. Now there are two metrics $g_i^{\mathrm{Flat}}$ and $g_j^{\mathrm{Flat}}$. Choose a partition of unity $\chi_i+\chi_j=1$, $|\nabla^m\chi_i|=O(r^{-m})$. Let
 $\tilde g=\chi_ig_i^{\mathrm{Flat}}+\chi_jg_j^{\mathrm{Flat}}$. It is a $\mathbb{T}^k$ invariant metric with $|\nabla^m \tilde g|=O(r^{-m-\epsilon})$. Note that there are still two maps from $M$ to the gluing $B_{\kappa r_i}(p_i)\times \mathbb{T}^k\cup_{\tilde f_{ij}}B_{\kappa r_j}(p_j)\times \mathbb{T}^k$: $\tilde f_{ij}\circ f_i$ and $f_j$. However, their distance is bounded by $O(r^{-\epsilon})$. For $r$ large enough, we can find out the unique $\tilde g$-minimal geodesic $\gamma$ satisfying $\gamma(0)=\tilde f_{ij}\circ f_i$ and $\gamma(1)=f_j$. Then $\gamma(\chi_j)$ gives a new map from $M$ to $B_{\kappa r_i}(p_i)\times \mathbb{T}^k\cup_{\tilde f_{ij}}B_{\kappa r_j}(p_j)\times \mathbb{T}^k$. Call that $\tilde f_i\cup \tilde f_j$.

 In conclusion, we have a $\mathbb{T}^k$-invariant metric $h$ on $$B_{\kappa r_i}(p_i)\times \mathbb{T}^k\cup_{\tilde f_{ij}}B_{\kappa r_j}(p_j)\times \mathbb{T}^k$$ and $$\tilde f_i\cup \tilde f_j:M\rightarrow B_{\kappa r_i}(p_i)\times \mathbb{T}^k\cup_{\tilde f_{ij}}B_{\kappa r_j}(p_j)\times \mathbb{T}^k$$
 with both $|\nabla^m h|=O(r^{-m-\epsilon})$ and $|\nabla^m(\tilde f_i\cup \tilde f_j)|=O(r^{1-m-\epsilon})$.

 After repeating everything for $(B_{\kappa r_i}(p_i)\times \mathbb{T}^k\cup_{\tilde f_{ij}}B_{\kappa r_j}(p_j)\times \mathbb{T}^k,\tilde g,\tilde f_i\cup \tilde f_j)$ and $(B_{\kappa r_l}(p_l)\times \mathbb{T}^k,g^{\mathrm{flat}}_l,f_l)$, we can get a new big chart. After $N$ times, we are done.
 \end{proof}

 \begin{theorem}
 Outside $K$, there is a $\mathbb{T}^k$-fiberation $E$ over $C(S(\infty))-B_R$ and a standard $\mathbb{T}^k$ invariant metric $h$ such that after the pull back by some diffeomorphism $h=g+O'(r^{-\epsilon})$.
 \label{standard-fiberation}
 \end{theorem}

 \begin{proof}
 The metric $\tilde g$ can be written as $$\sum_{i,j=1}^{4-k}a_{ij}(x)\mathrm{d}x_i\otimes\mathrm{d}x_j+
 \sum_{l=1}^{k}(\mathrm{d}\theta_l+\sum_{i=1}^{4-k}\eta_{li}(x)\mathrm{d}x_i)^2.$$
 The curvature of $a_{ij}$ belongs to $O'(r^{-2-\epsilon})$.
 By the result of Bando, Kasue and Nakajima \cite{BandoKasueNakajima}, there is a coordinate at infinity such that the difference between $a_{ij}$ and the flat metric on $C(S(\infty))-B_R$ belongs to $O'(r^{-\epsilon})$.
 So we can assume that $a_{ij}=\delta_{ij}$ without changing the condition $g=\tilde g+O'(r^{-\epsilon})$. Similarly, we can also replace $\eta_{lj}(x)$ by any standard connection form. As long as $\eta_{lj}$ is still in $O'(r^{-\epsilon})$, we still have $h=g+O'(r^{-\epsilon})$. Therefore, we only need to classify the torus fiberations over $C(S(\infty))-B_R$ topologically and give it a good enough standard metric $h$.

 (ALF-$A_k$)When $S(\infty)=\mathbb{S}^2$, the circle fiberation must be orientable. It is determined by the Euler class $e$.

 When $e=0$, we have the trivial product $(\mathbb{R}^3-B_R)\times \mathbb{S}^1$ as our standard model.

 When $e=\pm1$, we have the Taub-NUT metric with mass $m\not=0$:
 Let $$M_+=(\{(x_1,x_2,x_3)|x_1^2+x_2^2+x_3^2\ge R^2\}-\{(0,0,x_3)|x_3<0\})\times \mathbb{S}^1,$$
 $$M_-=(\{(x_1,x_2,x_3)|x_1^2+x_2^2+x_3^2\ge R^2\}-\{(0,0,x_3)|x_3>0\})\times \mathbb{S}^1$$
 Identify $(x_1,x_2,x_3,\theta_+)$ in $M_+$ with $(x_1,x_2,x_3,\theta_-+\mathrm{sign}(m)\mathrm{arg}(x_1+ix_2))$ in $M_-$.
 We get a manifold $M$.

 Let $r=\sqrt{x_1^2+x_2^2+x_3^2}, V=1+\frac{2m}{r},$
 \[\begin{split}\eta&=4|m|\mathrm{d}\theta_++4m\frac{(x_3-r)(x_1\mathrm{d}x_2-x_2\mathrm{d}x_1)}{2(x_1^2+x_2^2)r}\\
 &=4|m|\mathrm{d}\theta_-+4m\frac{(x_3+r)(x_1\mathrm{d}x_2-x_2\mathrm{d}x_1)}{2(x_1^2+x_2^2)r}.
 \end{split}\]
 Then the Taub-NUT metric with mass $m$ outside the ball $B_R$($R>>|m|$) is
 $$\mathrm{d}s^2=V\mathrm{d}\mathbf{x}^2+V^{-1}\eta^2$$
 with $$\mathrm{d}x_1=I^*(V^{-1}\eta)=J^*\mathrm{d}x_2=K^*\mathrm{d}x_3.$$

 There are lots of different conventions in the literatures. We use the convention from \cite{LeBrun}, but we compute the explicit form of $\eta$ using the formulas in \cite{HithcinKarlhedLindstromRocek}. When $m>0$, LeBrun \cite{LeBrun} proved that $M$ can be smoothly extended inside $B_R$ and becomes biholomorphic to $\mathbb{C}^2$. For $m<0$, the metric is only defined outside $B_R$, but it is enough for our purpose.

 There is a natural $\mathbb{Z}_{|e|}$ action on Taub-NUT metric by $\theta_\pm\rightarrow \theta_\pm+2\pi/|e|$ for $e=\pm1, \pm2, ...$
 The quotient of the Taub-NUT metric with positive mass $m$ by $\mathbb{Z}_{|e|}$ has Euler class $e<0$,
 The quotient of the Taub-NUT metric with negative mass $m$ by $\mathbb{Z}_{|e|}$ has Euler class $e>0$. Notice that the mass parameter $m$ is essentially a scaling parameter. Only the sign of $m$ determines the topology.

 Usually, people let $k=-e-1$ and call that a standard ALF-$A_k$ metric.

 (ALF-$D_k$)When $$S(\infty)=\mathbb{RP}^2=\{(x_1,x_2,x_3)\in\mathbb{S}^3|x_3\ge 0\}/(\cos t,\sin t,0)\sim(-\cos t,-\sin t,0)$$ topologically,
 the fiberation is the trivial fiberation over the disc after identifying $(\cos t,\sin t,0,\theta)$ with $(\cos(t+\pi),\sin(t+\pi),0,f(t)-\theta)$. So $f(\pi)-f(0)=-2e\pi$. The integer $e$ determines the topological type.

 When $e=0$, we have the trivial product $(\mathbb{R}^3-B_R)\times \mathbb{S}^1$ after identifying $(\mathbf{x},\theta)$ with $(-\mathbf{x},-\theta)$ as our standard model.

 When $e$ is nonzero, it is the quotient of the Taub-NUT metric outside $B_R$ by the binary dihedral group $D_{4|e|}=\{\sigma,\tau|\sigma^{2|e|}=1,\sigma^{|e|}=\tau^2,\tau\sigma\tau^{-1}=\sigma^{-1}\}$ which acts by $\sigma(\mathbf{x},\theta_\pm)=\sigma(\mathbf{x},\theta_\pm+\pi/|e|)$ and  $\tau(\mathbf{x},\theta_+)=(-\mathbf{x},\theta_-=-\theta_+)$ from $M_+$ to $M_-$ with $\tau(\mathbf{x},\theta_-)=(-\mathbf{x},\theta_+=\pi-\theta_-)$ from $M_-$ to $M_+$.
 When the mass is positive, $e$ is negative. When the mass is negative, $e$ is positive.

 Usually, people let $k=-e+2$ and call that a standard ALF-$D_k$ metric.

 (ALG)When $S(\infty)=\mathbb{S}^1$, the topological type is determined by the monodramy. In other words, when we travel along $S(\infty)$, there is some rotation but the lattice $\Lambda=\mathbb{Z}|v_1|\oplus\mathbb{Z}\tau|v_1|$ is still invariant. So we have the equation
 \[
 \left( {\begin{array}{*{20}c}
    a & b  \\
    c & d   \\
 \end{array}} \right)
 \left( {\begin{array}{*{20}c}
    1   \\
    \tau \\
 \end{array}} \right)
 =
 \left( {\begin{array}{*{20}c}
    e^{i\theta} & 0  \\
    0 & e^{i\theta} \\
 \end{array}} \right)
 \left( {\begin{array}{*{20}c}
    1   \\
    \tau \\
 \end{array}} \right).
 \]
 for some
 \[
 \left( {\begin{array}{*{20}c}
    a & b  \\
    c & d   \\
 \end{array}} \right)\in \mathrm{GL}(2,\mathbb{Z}).
 \]
 So
 \[
 0=\det\left( {\begin{array}{*{20}c}
    a-e^{i\theta} & b  \\
    c & d-e^{i\theta} \\
 \end{array}} \right)=
 ad-bc-(a+d)e^{i\theta}+(e^{i\theta})^2.
 \]
 Except the case where $e^{i\theta}=\pm1$,we have $\Delta=(a+d)^2-4(ad-bc)<0$.
 So $ad-bc>0$, it must be 1 to make sure the matrix invertible. So $a+d=0$ or $a+d=\pm1$.
 The quadratic equation $e^{i \theta}$ satisfies must be one of the following
 $$x^2 + x + 1 = 0,\qquad x^2 - x + 1 =0,\qquad {\rm and} \qquad  x^2 + 1 = 0.$$
 We can solve $e^{i\theta}$ accordingly:

 $$\frac{-1\pm i \sqrt{3} }{2}  =  e^{i \frac{2\pi}{3}}, e^{i \frac{4\pi}{3}}; \qquad \frac{1\pm i \sqrt{3} }{2}  = e^{i \frac{\pi}{3}}, e^{i \frac{5\pi}{3}}; \qquad \pm i= e^{i \frac{\pi}{2}}, e^{ i \frac{3\pi}{2}}.$$

 Therefore, the rotation angle $\theta=2\pi\beta$ and the lattice $\Lambda=\mathbb{Z}|v_1|\oplus\mathbb{Z}\tau|v_1|$ are in the following list: (We may replace $\tau$ by something like $\tau-1$, but that will not change the lattice at all)

 (Regular) $\mathrm{Im}\tau>0$, $\beta=1$.

 (I$_0^*$) $\mathrm{Im}\tau>0$, $\beta=1/2$.

 (II) $\tau=e^{2\pi i/3}$, $\beta=1/6$.

 (II$^*$) $\tau=e^{2\pi i/3}$, $\beta=5/6$.

 (III) $\tau=i$, $\beta=1/4$.

 (III$^*$) $\tau=i$, $\beta=3/4$.

 (IV) $\tau=e^{2\pi i/3}$, $\beta=1/3$.

 (IV$^*$) $\tau=e^{2\pi i/3}$, $\beta=2/3$.

 Note that they all correspond to Kodaira's classification of special fibers of elliptic surface in \cite{Kodaira}!
 If we identify $(u,v)$ with $(e^{2\pi i \beta}u,e^{-2\pi i \beta}v)$ in the space $\{(u,v)|\mathrm{arg} u\in[0,2\pi\beta],|u|\ge R\}\subset(\mathbb{C}-B_R)\times \mathbb{C}/(\mathbb{Z}|v_1|\oplus\mathbb{Z}\tau|v_1|)$, we have the standard flat hyperk\"ahler metric $h=\frac{i}{2}(du\wedge d\bar u+dv\wedge d\bar v)$.
 Note that $\mathrm{SU}(2)$ is transitive, so we can choose the complex structure $a_1I+a_2J+a_3K$ properly so that $\bar\partial_g=\bar\partial_h+O(r^{-\epsilon})\nabla_h$.

 (ALH-non-splitting)When $C(S(\infty))=\mathbb{R}_+$, $h$ can be simply chosen to be the product metric of $[R,\infty)$ and a flat 3-torus.

 \end{proof}

\section{The construction of holomorphic functions }

 In this section, we prove Theorem 2 and Theorem 3. Our goal in this section is to construct global holomorphic functions on gravitational instantons $M$ with prescribed growth order. It is usually very hard to do so directly. However, it is much easier to construct holomorphic functions on the standard models $(E,h)$ first. Then it can be pulled back to $(M,g)$ and cut off to obtain an almost holomorphic function $f$ on $M$. To get rid of the error, we can solve the $\bar\partial$ equation
 $$\bar\partial g=\bar\partial f$$
 for $g$ much smaller than $f$. If successful, then $f-g$ will be the required function.
 Unfortunately, it is hard for us to solve $g$ directly. So instead, we solve the equation
 $$-(\bar\partial\bar\partial^*+\bar\partial^*\bar\partial)\phi=\bar\partial f.$$
 The order of $\bar\partial^*\phi$ and $\bar\partial\phi$ will be smaller than the order of $f$ if we solve $\phi$ properly.
 Notice that there is a covariant constant (0,2)-form $\omega^-$, so the harmonic (0,2)-form $\bar\partial\phi$ is essentially  a harmonic function. Generally speaking, the order of growth of harmonic functions on $M$ is the same as the harmonic functions on $E$. So if we get $f$ from the smallest nonconstant harmonic function on $E$, we expect $\bar\partial\phi$ to be 0. Therefore $f+\bar\partial^*\phi$ will be the required global holomorphic function on $M$.

 To solve the Laplacian equation for (0,1)-forms, we need some elliptic estimates. The ALH case requires more care.  To obtain a good estimate in ALH case, we need to prove the exponential decay of curvature first. This is feasible after we develop some elliptic estimates for the Riemannian curvature tensor.

 Therefore, in the first two subsections, we develop the elliptic estimates for tensors on a manifold $M$ asymptotic to the standard model. We would like to work on both forms and the curvature tensors on general $M$ which may not be hyperk\"ahler. Therefore, we always use the Bochner Laplacian $-\nabla^*\nabla$ in order to apply the Bochner techniques. For gravitational instantons, the Weitzenb\"ock formula tells us that the Bochner Laplacian equals to the operator $-(\bar\partial\bar\partial^*+\bar\partial^*\bar\partial)$ for functions and (0,1)-forms. Then in the third subsection, we use the mentioned technique to construct global holomorphic functions on ALF and ALG instantons. In the fourth subsection, we use this estimate to prove the exponential decay of curvature of ALH-instantons. This allows us to develop an elliptic estimate with exponential growth weights in the fifth subsection. In the sixth subsection, we use the same method to construct global holomorphic functions on ALH instantons.
 In the last two subsections, we make use of the global holomorphic functions to prove our second and third main theorems.

 Analysis in weighted Hilbert space is well studied and perhaps some estimates in this section are already known to experts \cite{Hormander} \cite{HauselHunsickerMazzeo} \cite{MinerbeMass}. However, to avoid problems caused by subtle differences between different settings, we instead give a self-contained proof.

\subsection{Weighted Hilbert space}
 In this subsection, we do some technical preparations.  We will use the following weighted Hilbert spaces: (Please notice the change of the meaning of $r$ as in the end of the second section.)

 \begin{definition}

 Define the ${L^2_{\delta}}$-norm of a tensor by
 $$||\phi||_{L^2_{\delta}}=\sqrt{\int_M|\phi|^2r^\delta\mathrm{dVol}}.$$
 Let $L^2_{\delta}$ be the space of tensors with finite $L^2_{\delta}$-norm.
 Define $\nabla\phi=\psi$ in the distribution sense if for any $\xi\in\mathrm{C}^{\infty}_0$, we have $(\phi,\nabla^*\xi)=(\psi,\xi)$. Let $H^2_{\delta}$ be the space of  all tensors $\phi$
 such that
 $$\phi \in L^2_{\delta},  \qquad  \nabla \phi \in L^2_{\delta+2}\qquad {\rm and }\qquad   \nabla^2 \phi \in L^2_{\delta+4}.$$
 We can define the norm in this weighted space by
 $$||\phi||_{H^2_{\delta}}=\sqrt{\int_M|\phi|^2r^\delta\mathrm{dVol}+\int_M|\nabla\phi|^2r^{\delta+2}\mathrm{dVol}
 +\int_M|\nabla^2\phi|^2r^{\delta+4}\mathrm{dVol}}.$$
 The inner product is defined accordingly.
 \end{definition}

 \begin{proposition}
 For any $\delta$, $H^2_{\delta}$ is a Hilbert space and the space of compactly supported smooth tensors $\mathrm{C}^{\infty}_0$ is dense.
 \end{proposition}

 \begin{proof}
 The map $\phi\rightarrow\phi r^{\delta/2}$ defines an isometry between $L^2_{\delta}$ and $L^2$. Since $L^2$ is complete, $L^2_{\delta}$ is also complete. Now if $|\phi_i-\phi_j|_{H^2_{\delta}}\rightarrow 0$, then both $|\phi_i-\phi_j|_{L^2_{\delta}}$ and $|\nabla^m\phi_i-\nabla^m\phi_j|_{L^2_{\delta+2m}}$ go to 0, $m=1,2$. By completeness, $\phi_i$ converge to $\phi$ in $L^2_{\delta}$, and $\nabla\phi_i$ converge to $\psi$ in $L^2_{\delta+2}$. Now pick any test tensor $\xi\in\mathrm{C}^{\infty}_0$,
 $$(\phi,\nabla^*\xi)=\lim_{i\rightarrow \infty}(\phi_i,\nabla^*\xi)=\lim_{i\rightarrow \infty}(\nabla\phi_i,\xi)=(\psi,\xi).$$
 So $\nabla\phi=\psi$ in the distribution sense. The second derivative is similar. So $\phi_i$ converge to $\phi$ in $H^2_{\delta}$, too.

 For the density, let $\chi_R=\chi(r/R)$. Then
 \[\begin{split}|\phi-\phi\chi(r/R)|_{H^2_{\delta}}\le C \int_M(|(1-\chi_R)\phi|^2r^\delta+|(1-\chi_R)\nabla\phi|^2r^{\delta+2}+|\nabla\chi_R||\phi|^2r^{\delta+2}
 \\+|\nabla^2\chi_R \phi|^2r^{\delta+4}+|(1-\chi_R)\nabla^2\phi|^2r^{\delta+4}
 +|\nabla\chi_R \nabla\phi|r^{\delta+4})
 \end{split}\]
 So $\chi(r/R)\phi$ converge to $\phi$ in $H^2_{\delta}$ when $R$ goes to infinity since $|\nabla\chi_R|\le C/R$ and $|\nabla^2\chi_R|\le C/R^2$.
 Now the standard convolution method implies the density of $\mathrm{C}^{\infty}_0$.

 \end{proof}

 \begin{lemma}
 For any harmonic tensor $\phi$ in $H^2_{\delta}$ and any large enough $r$,
 $$|\phi(y)|\le C||\phi||_{H^2_{\delta}}r(y)^{-\delta/2+k/2-2}.$$
 When $-\delta/2+k/2-2<0$, $\phi=0$.
 \label{harmonic-form}
 \end{lemma}

 \begin{proof}
 Given $y\in M$, suppose $r(y)=20R$. Then the ball $B_{2R}(y)\subset M$ is asymptotic to $B_{2R}(0)\times \mathbb{T}^k\subset \mathbb{R}^{4-k}\times \mathbb{T}^k$. Consider the covering space $\mathbb{R}^4$ of $\mathbb{R}^{4-k}\times \mathbb{T}^k$. If we apply Gilbarg and Trudinger's Theorem 9.20 in \cite{GilbargTrudinger} there, we would get $$|\phi|^2(y)\le \frac{C}{|B_{2R}(0)|}\int_{B_{2R}(0)}|\phi|^2.$$
 So $$|\phi(y)|\le C||\phi||_{H^2_{\delta}}r(y)^{-\delta/2+k/2-2}.$$
 Now the maximal principle implies the last result in the lemma because $\Delta|\phi|^2=2|\nabla\phi|^2\ge 0$.
 \end{proof}

 Now we need an weighted $L^2$-estimate.

 \begin{lemma}
 For the standard ALF, ALG or ALH metric in Theorem \ref{standard-fiberation}, suppose $\phi$ is a smooth form supported in $B_{\tilde R}-B_R$. Then as long as $R$ is large enough, $$\int_E|\nabla^2\phi|^2r^{\delta+4}+\int_E|\nabla\phi|^2r^{\delta+2}\le C(\int_E|\Delta\phi|^2r^{\delta+4}+\int_E|\phi|^2r^{\delta})$$
 \label{weighted-L2-estimate}
 \end{lemma}

 \begin{proof}
 We only need to prove the same thing on $B_{\kappa r_j}(p_j)\subset E$ uniformly. It is enough to consider the covering $B_{\kappa r_j}(0)\subset \mathbb{R}^4$. Notice that $h=$flat metric+$O'(r^{-1})$. So we can simply use the Theorem 9.11 of \cite{GilbargTrudinger}.
 \end{proof}

\subsection{Elliptic estimates with polynomial growth weights}
 In this subsection, we will prove the main estimate for tensors in the weighted Hilbert space with polynomial growth weights.

 We started the estimate for functions on $\mathbb{R}^d$. Then we extend this to $\mathbb{T}^k$ invariant tensors. We can improve it to general tensors on the standard fiberation $E$. Then we can transfer that estimate back to any manifold $M$ asymptotic to the standard model. This main estimate allows us to prove the solvablity of Bochner Laplacian equation for  tensors.

 \begin{theorem}
 Suppose $f$ is a real smooth function on $\mathbb{R}^d$($d=1,2,3,...$) supported in an annulus,$\;\;\delta$ is not an integer. Then
 $$\int_{\mathbb{R}^d}|f|^2r^{\delta}\mathrm{dVol}<C\int_{\mathbb{R}^d}|\Delta f|^2r^{\delta+4}\mathrm{dVol}.$$
 \label{estimate-on-Rd}
 \end{theorem}

 \begin{proof} For the Laplacian on the standard sphere $\mathbb{S}^{d-1}, $  it is well know it has eigenfunctions $\phi_{j,l}$ with eigenvalue $-j(d-2+j)$, $l=1,2,...,n_j$. (For $d=1$,all $n_j$ are 0 except $n_0=1$ and $\phi_{j,1}=1$). We write $f$ in terms of those eigenfunctions $$f\sim\sum_{j=0}^{\infty}\sum_{l=1}^{n_j} f_{j,l}(r)\phi_{j,l}(\theta),$$
 where $$f_{j,l}(r)=\int_{\mathbb{S}^{d-1}} f(r,\theta)\phi_{j,l}(\theta)\mathrm{dVol}.$$
 Then $$ \begin{array}{lcl}
  \Delta f & \sim & \sum_{j=0}^{\infty}\sum_{l=1}^{n_j}(f''_{j,l}+\frac{d-1}{r}f'_{j,l}-\frac{j(d-2+j)}{r^2}f_{j,l})\phi_{j,l}(\theta) \\
  & = & \sum_{j=0}^{\infty}\sum_{l=1}^{n_j}r^{-j-d+1}[r^{2j+d-1}(r^{-j}f_{j,l})']'\phi_{j,l}(\theta).\end{array} $$
 From integral by parts and the Cauchy-Schwartz inequality
 $$ \begin{array}{lcl} (\int_0^\infty g^2r^{\mu}\mathrm{d}r)^2 & = & \left(\frac{-2}{\mu+1}\int_0^\infty gg'r^{\mu+1}d\,r\right)^2 \\
 & \le &  \frac{4}{(\mu+1)^2}\int_0^\infty g^2r^{\mu}\mathrm{d}r\int_0^\infty (g')^2r^{\mu+2}\mathrm{d}r. \end{array} $$
 So we get the Hardy's inequality
 $$\int_0^\infty g^2r^{\mu}\mathrm{d}r\le\frac{4}{(\mu+1)^2}\int_0^\infty (g')^2r^{\mu+2}\mathrm{d}r.$$
 Therefore $$\begin{array}{lcl} \int_0^\infty f_{j,l}^2r^{\delta}r^{d-1}\mathrm{d}r & = & \int_0^\infty (r^{-j}f_{j,l})^2r^{\delta+2j+d-1}\mathrm{d}r\\
 & \le & \frac{4}{(\delta+2j+d)^2}\int_0^\infty[(r^{-j}f_{j,l})']^2r^{\delta+2j+d+1}\mathrm{d}r\\ & = &
 \frac{4}{(\delta+2j+d)^2}\int_0^\infty[r^{2j+d-1}(r^{-j}f_{j,l})']^2r^{\delta-2j-d+3}\mathrm{d}r\\
 &\le& \frac{16\int_0^\infty(r^{-j-d+1}[r^{2j+d-1}(r^{-j}f_{j,l})']')^2r^{\delta+4}r^{d-1}\mathrm{d}r}{(\delta+2j+d)^2(\delta-2j-d+4)^2}.
 \end{array}$$
 By Fubini Theorem and the Hilbert-Schmidt Theorem (When $d=2$, we get exactly the Fourier series, so the Hilbert-Schmidt theorem is reduced to the Parseval's identity) as long as $\delta$ is not an interger, we are done.
 \end{proof}

 \begin{theorem}
 Suppose $(E,h)$ is the product of $[R,\infty)$ and $\mathbb{T}^3$, $\phi$ is a smooth $\mathbb{T}^3$-invariant tensor supported in $B_{\tilde R}-B_R$. Then as long as $\delta$ is not an integer, for large enough $R$,
 $$ \int_{E}|\phi|^2r^{\delta}\mathrm{dVol}<C \int_{E}|\Delta\phi|^2r^{\delta+4}\mathrm{dVol}$$
 \label{estimate-of-trivial-products}
 \end{theorem}

 \begin{proof}
 Since the tangent bundle is trivial, the estimate of tensors is reduced to the estimate of their coefficients, which has been proved in Theorem \ref{estimate-on-Rd}.
 \end{proof}

 \begin{theorem}
 Suppose $(E,h)$ is the standard ALG metric as in Theorem \ref{standard-fiberation}, $\phi$ is a smooth $\mathbb{T}^2$-invariant tensor supported in $B_{\tilde R}-B_R$. Then as long as $30\delta$ is not an integer, for large enough $R$,
 $$ \int_{E}|\phi|^2r^{\delta}\mathrm{dVol}<C \int_{E}|\Delta\phi|^2r^{\delta+4}\mathrm{dVol}.$$
 \label{estimate-of-flat-elliptic-surface}
 \end{theorem}

 \begin{proof}
 Let $\beta=\frac{m}{n}$. Then it is enough to do the same estimate on the $n$-fold covering $\tilde E-B_R$ of $E-B_R$. $\tilde E-B_R$ is the isometric product of the $m$-fold covering of $\mathbb{C}-B_R$ and $\mathbb{T}^2$. So it is enough to prove Theorem \ref{estimate-on-Rd} on the $m$-fold cover of $\mathbb{C}-B_R$. If we write $f\sim\sum_{j=-\infty}^{\infty}f_j(r)e^{i\theta j/m}$, where $\theta\in[0,2m\pi]$ then all the works in the proof of Theorem \ref{estimate-on-Rd} go through except that we have to replace $j$ by $j/m$ there. So as long as $m\delta$ is not an integer, we are done. ($m=1,2,3,5$)
 \end{proof}

 \begin{theorem}
 Suppose $(E,h)$ is the standard ALF metric as in Theorem \ref{standard-fiberation}, $\phi$ is a smooth $\mathbb{S}^1$-invariant tensor supported in $B_{\tilde R}-B_R$. Then as long as $\delta$ is not an integer, for large enough $R$,
 $$ \int_{E}|\phi|^2r^{\delta}\mathrm{dVol}<C \int_{E}|\Delta\phi|^2r^{\delta+4}\mathrm{dVol}$$
 \label{estimate-of-Taub-NUT-metric}
 \end{theorem}

 \begin{proof}
 By Theorem \ref{standard-fiberation}, it is enough to consider the trivial product of $\mathbb{R}^3$ and $\mathbb{S}^1$ or the Taub-NUT metric with nonzero mass $m$.
 We use 1-forms as example, the proof for general tensors is similar.
 In the trivial product case, we can write any form as $A\mathrm{d}x_1+B\mathrm{d}x_2+C\mathrm{d}x_3+D\mathrm{d}\theta$.
 In the remaining cases, any form can be written as $A\mathrm{d}x_1+B\mathrm{d}x_2+C\mathrm{d}x_3+D\eta$.
 In each case we get 4 functions on $\mathbb{R}^3-B_R$ which can be filled in by 0 on $B_R$ to get smooth functions on $\mathbb{R}^3$. So we can apply Theorem \ref{estimate-on-Rd} to them. Since the Taub-NUT metric is the flat metric with error $O'(r^{-1})$, while $\eta=\mathrm{d}\theta+O'(r^{-1})$ locally, by Lemma \ref{weighted-L2-estimate}, we can get our estimate as long as $R$ is large enough.
 \end{proof}

 \begin{theorem}
 Suppose $(E,h)$ is the standard ALF, ALG, or ALH-non-splitting metric in Theorem \ref{standard-fiberation}, $\phi$ is a smooth tensor supported in $B_{\tilde R}-B_R$. Then as long as $30\delta$ is not an integer, for large enough $R$,
 $$ \int_{E}|\phi|^2r^{\delta}\mathrm{dVol}<C \int_{E}|\Delta\phi|^2r^{\delta+4}\mathrm{dVol}$$
 \label{estimate-of-h}
 \end{theorem}

 \begin{proof}
 First average $\phi$ on each $\mathbb{T}^k$ (k=1,2,3) to get an invariant tensor $\phi_0$. Then we only need to get some estimates of the $\phi-\phi_0$ part. It is enough to prove that in each $B_{\kappa r_i}(p_i)\subset E$,
 $$\int_{B_{\kappa r_i}(p_i)}|\phi-\phi_0|^2\mathrm{dVol}<C \int_{B_{\kappa r_i}(p_i)}|\Delta(\phi-\phi_0)|^2\mathrm{dVol}$$
 for a uniform constant $C$ and any tensor $\phi$ supported in $B_{\kappa r_i}(p_i)\subset E$ because then we can use the partition of unity and move every error term to the left hand side by Lemma \ref{weighted-L2-estimate}. Again, we may cancel error terms and assume that the metric is flat. So the estimate of forms is reduced to functions which are the coefficients of the forms. Standard Poincar\'e inequality on torus implies that $$(\int_{B_{\kappa r}\times\mathbb{T}^k}|f-f_0|^2)^2
 \le C(\int_{B_{\kappa r}\times\mathbb{T}^k}|\nabla_{\mathbb{T}^k}(f-f_0)|^2)^2\le C(\int_{B_{\kappa r}\times\mathbb{T}^k}|\nabla(f-f_0)|^2)^2$$
 $$=C(\int_{B_{\kappa r}\times\mathbb{T}^k}(f-f_0)\Delta(f-f_0))^2
 \le C\int_{B_{\kappa r}\times\mathbb{T}^k}|f-f_0|^2\int_{B_{\kappa r}\times\mathbb{T}^k}|\Delta(f-f_0)|^2,$$
 where $\nabla_{\mathbb{T}^k}$ means the partial derivative with respect to the fiber direction.
 So we are done when $R$ is large enough.
 \end{proof}

 \begin{lemma}
 Suppose $X$,$Y$,$Z$ are Banach spaces,
 $D:X\rightarrow Y$, $i:X\rightarrow Z$ are bounded linear operators, $i$ is compact. Suppose $||\phi||_X\le C(||D\phi||_Y+||i\phi||_Z)$. Then as long as $\mathrm{Ker}D=\{0\}$, we have $||\phi||_X\le C||D\phi||_Y$.
 \label{functional-analysis}
 \end{lemma}

 \begin{proof}
 Suppose the estimate does not hold, then there are $\phi_k$ satisfying $||\phi_k||=1$,
 but $||D\phi_k||\rightarrow 0$. By the compactness of $i$, we know that $||i\phi_k-i\phi_l||_Z\rightarrow$0.
 So $$||\phi_k-\phi_l||_X\le C(||D\phi_k-D\phi_l||_Y+||i\phi_k-i\phi_l||_Z)\rightarrow 0.$$
 So $\phi_k\rightarrow \phi_{\infty}$. But then $D\phi_k\rightarrow D\phi_{\infty}$, $D\phi_{\infty}=0$, $\phi_{\infty}\in\mathrm{Ker}D$, contradiction.
 \end{proof}

 \begin{theorem}
 Suppose $M$ is asymptotic to the standard ALF, ALG or ALH-non-splitting model, then for any tensor $\phi\in H^2_{\delta}(M)$, as long as $30\delta$ is not an integer and $-\delta/2-2+k/2<0$, we have
 $$||\phi||_{H^2_{\delta}(M)}<C\int_{M}|\Delta\phi|^2r^{\delta+4}\mathrm{dVol}.$$
 \label{estimate-of-g}
 \end{theorem}

 \begin{proof}
 It is enough to prove everything for $\mathrm{C}^{\infty}_0$. Note that $$\Delta_g\phi=\Delta_h\phi+O(r^{-\epsilon})|\nabla^2\phi|+O(r^{-\epsilon-1})|\nabla\phi|+O(r^{-\epsilon-2})|\phi|.$$
 After applying Theorem \ref{estimate-of-h} and Lemma \ref{weighted-L2-estimate}, we know that the estimate holds as long as $\phi$ is 0 inside a big enough ball $B_R$. For general $\phi$, we can apply the estimate to the form $(1-\chi(r/R))^2\phi$. So
 $$||\phi||_{H^2_{\delta}(M)}<C(\int_{M}|\Delta\phi|^2r^{\delta+4}\mathrm{dVol}+||\phi||_{H^2(B_{2R})})$$
 $$<C(\int_{M}|\Delta\phi|^2r^{\delta+4}\mathrm{dVol}+\int_{B_{4R}}|\phi|^2)$$
 by Theorem 9.11 of \cite{GilbargTrudinger}.
 By Lemma \ref{functional-analysis} and Rellich's lemma, it is enough to prove that Ker$\Delta$=\{0\}.
 This follows from Lemma \ref{harmonic-form}.
 \end{proof}

 \begin{theorem}
 Suppose $30\delta$ is not an integer and $-\delta/2-2+k/2<0$. For any $\phi\in L^2_{-\delta}(M)$, there exists a tensor $\psi\in H^2_{-\delta-4}(M)$ such that $\Delta\psi=\phi$.
 \label{solve-laplacian-equation}
 \end{theorem}

 \begin{proof}
 Consider the Laplacian operator $\Delta:L^2_{-\delta-4}\rightarrow L^2_{-\delta}$.
 The formal adjoint is then $\Delta^*\phi=r^{\delta+4}\Delta(r^{-\delta}\phi)$.
 Apply Theorem \ref{estimate-of-g} to $r^{-\delta}\phi$,
 $$C^{-1}||\Delta^*\phi||_{L^2_{-\delta-4}}\le||\phi||_{H^2_{-\delta}}=||r^{-\delta}\phi||_{H^2_{\delta}}\le C||\Delta^*\phi||_{L^2_{-\delta-4}}.$$
 So $\Delta^*$ has closed range. Now $$|(\phi,\theta)_{L^2_{-\delta}}|\le||\phi||_{L^2_{-\delta}}||\theta||_{L^2_{-\delta}}\le C||\phi||_{L^2_{-\delta}}||\Delta^*\theta||_{L^2_{-\delta-4}},$$ so $\Delta^*\theta\rightarrow(\phi,\theta)_{L^2_{-\delta}}$ defines a bound linear function in the range of $\Delta^*$. By Riesz representation theorem, there exists $\psi\in\mathrm{Im}(\Delta^*)$ such that $(\psi,\Delta^*\theta)_{L^2_{-\delta-4}}=(\phi,\theta)_{L^2_{-\delta}}$. Now we get the theorem from the standard elliptic regularity theory.
 \end{proof}

\subsection{Holomorphic functions on ALF and ALG instantons}

 After proving the main estimate in the last subsection, we are ready to prove the existence of global holomorphic functions on both ALF and ALG instatons. Our first theorem deal with the growth order of harmonic functions on $M$

 \begin{theorem}
 Suppose $M$ is asymptotic to the standard ALF or ALG model. Given any harmonic function $f\in L^2_{\delta}(M)$ for some $\delta$, there exist an $\gamma$ such that $f$ is $O(r^\gamma)$ but not $o(r^\gamma)$. What is more, when $C(S(\infty))=\mathbb{C}_\beta$(ALG), $\beta\gamma$ must be an integer. When $C(S(\infty))=\mathbb{R}^3$(ALF-$A_k$), $\gamma$ must be an integer. When $C(S(\infty))=\mathbb{R}^3/\mathbb{Z}_2$(ALF-$D_k$), $\gamma$ must be an even number.
 \label{growth-rate}
 \end{theorem}

 \begin{proof}
 $f$ also belongs to $L^2_{\delta'}(M)$ for other $\delta'$. Without of loss of generality, assume $\delta$ is bigger than the superior of those $\delta'$ minus $\epsilon$. The superior exists because of the vanishing part of Theorem \ref{harmonic-form}. By Lemma \ref{weighted-L2-estimate}, $f\in H^2_{\delta}(M)$. Cut off $f$ so that it vanish inside a large ball $B_R$. Move this function to $E$. Then $\Delta(f(1-\chi(r/R))\in L^2_{\delta+4+\epsilon}$. Decompose $f(1-\chi(r/R))$ into $\mathbb{T}^k$-invariant part $f_0$ and the perpendicular part $f_1$.

 Then $f_1$ is much smaller than the growth rate of $f(1-\chi)$. Without loss of generality, we can assume that $f(1-\chi(r/R))$ is invariant.

 Now again, we can transfer this invariant function to the tangent cone at infinity $C(S(\infty))$. When $C(S(\infty))=\mathbb{R}^3/\mathbb{Z}_2$ (ALF-$D_k$), we get a function $\tilde f$ on its double cover $\mathbb{R}^3$ naturally. When $C(S(\infty))=\mathbb{C}_\beta$ (ALG), we get a function $\tilde f(z)$ on $\mathbb{C}=\mathbb{R}^2$ defined by $\tilde f=(f(1-\chi(r/R)))(z^\beta)$. Again the growth rate of $\Delta(\tilde f)$ is at most the growth rate of $\tilde f$ minus two then minus $\epsilon$, so we can find out a function $\psi$ with growth rate the rate of $\tilde f$ minus $\epsilon$ such that $\Delta\psi=\Delta(\tilde f)$. So $\tilde f-\psi$ becomes a harmonic function on $\mathbb{R}^3$ or $\mathbb{R}^2$. The gradient estimate implies that after taking derivatives for some times, we get 0. In other word $\tilde f-\psi$ must be a polynomial. So the growth rate must be integer. For the $C(S(\infty))=\mathbb{R}^3/\mathbb{Z}_2$(ALF-$D_k$) case, we may replace $\psi(x)$ by $(\psi(x)+\psi(-x))/2$ so that it is invariant under the $\mathbb{Z}_2$ action. So the polynomial must have even degree.
 \end{proof}

 Now we can prove the existence of global holomorphic function on ALG gravitational instantons.

 \begin{theorem}
 There exists a global holomorphic function on any ALG gravitational instanton $M$ such that any far enough fiber is biholomorphic to a complex torus.
 \label{ALG}
 \end{theorem}

 \begin{proof}
 In this case $k=2$. By theorem \ref{standard-fiberation}, the metric near infinity is asymptotic to the elliptic surface $(E,h)$. For $(E,h)$, $u^{1/\beta}$ is a well defined holomorphic function outside $B_R$. Now if we pull back  $u^{1/\beta}$ from the elliptic surface, cut it off and fill in with 0 inside $K$, we obtain a function $f$ satisfying $$\bar\partial_g f=\phi=O(r^{1/\beta-1-\epsilon}).$$
 Pick any small positive number  $\delta\in(\max\{-2,2/\beta-2\epsilon\},2/\beta-\epsilon)$, such that $30\delta$ is not an integer. Thus, $\phi\in L^2_{-\delta}$. By Theorem \ref{solve-laplacian-equation}, there exists $\psi\in H^2_{-\delta-4}$ such that
 $$\phi=\Delta\psi=-(\bar\partial^*\bar\partial+\bar\partial\bar\partial^*)\psi$$
 in the distribution sense. Elliptic regularity implies that $\psi$ is a smooth $(0,1)$-form. Take $\bar\partial$ on both side of this equation. Notice that $\bar\partial\phi = 0.\;$ Thus
 \[
 0=-\bar\partial\bar\partial^*(\bar\partial\psi)=\Delta(\bar\partial\psi).\]  By Lemma \ref{harmonic-form},   $\bar\partial\psi = O(r^{1/\beta-\epsilon/2})$. We can write this (0,2) form as $\xi\omega^+$, where $\omega^+$ is the parallel (0,2)-form. Then $\xi$ is a harmonic function. By Theorem \ref{growth-rate}, it is constant. Therefore $\bar\partial(f+\bar\partial^*\psi)=0$. So $f+\bar\partial^*\psi$ is a global holomorphic function. After analyzing the growth rate, we can also show that $|\mathrm{d}\bar\partial^*\psi|<<|\mathrm{d}f|$ for large $r$. So the fiber far from origin is an compact Riemann surface with genus 1. It must be a complex torus by the uniformization theorem.
 \end{proof}

 Similarly, we can prove

 \begin{theorem}
 There exists a global holomorphic function on any ALF-$D_k$ gravitational instanton $M$.
 \label{ALF-Dk}
 \end{theorem}

 \begin{proof}
 $M$ is asymptotic to a fiberation over $\mathbb{R}^3/\mathbb{Z}_2=\mathbb{R}^3/\mathbf{x}\sim-\mathbf{x}$. The function $(x_2+ix_3)^2$ is well defined over $E$. The proof of the last theorem will produce a global holomorphic function in ALF-$D_k$ case.
 \end{proof}

 The existence of global holomorphic function on any ALF-$A_k$ gravitational instanton $M$ can also be proved by the same way. Actually, Minerbe had a simpler proof in \cite{MinerbeMultiTaubNUT}. It is an essential step in his classification of ALF-$A_k$ instantons.

 \subsection{Exponential decay of curvature of ALH instantons}

 For ALH-non-splitting instantons, there is a self-improvement forcing the curvature to decay exponentially. Therefore, the metric must converge to the flat one exponentially.

 \begin{proposition}
 If the Ricci curvature is 0, then
 $$\Delta R_{ijkl}=Q(\mathrm{Rm}).$$
 \end{proposition}

 \begin{proof}
 $$\Delta R_{ijkl}=R_{ijkl,m}\,^m=-R_{ijlm,k}\,^m-R_{ijmk,l}\,^m$$
 $$=-R_{ijlm}\,^{,m}\,_k-R_{ijmk}\,^{,m}\,_l+Q(\mathrm{Rm})$$
 By Bianchi identity and the vanishing of the Ricci curvature,
 $$R_{ijlm}\,^{,m}=R_{lmij}\,^{,m}=-R_{lmj}\,^{m}\,_{,i}-R_{lm}\,^{m}\,_{i,j}=0.$$
 Similarly $$R_{ijmk}\,^{,m}=0.$$
 So we get the conclusion.
 \end{proof}

 \begin{theorem}
 In the ALH-non-splitting case, there exists a constant $\mu$ such that the Riemannian curvature at $p$ is bounded by $Ce^{-\mu r(p)}$.
 \label{exponential-decay-of-curvature}
 \end{theorem}

 \begin{proof}
 Pull back the Riemannian curvature tensor of $g$ to $([R,\infty)\times \mathbb{T}^3,h)$, where $h$ is the standard flat metric , we get a tensor $\mathbf{T}$ satisfying the equation $\mathbf{DT}=0$, where $\mathbf{D}=\mathbf{A}^{ij}\nabla_i\nabla_j+\mathbf{B}^{i}\nabla_i+\mathbf{C}$ is a tensor-valued second order elliptic operator such that
 $$|\mathbf{A}^{ij}-\delta^{ij}\mathbf{Id}|\le Cr^{-\epsilon},|\mathbf{B}^{i}|\le Cr^{-1-\epsilon},|\mathbf{C}|\le Cr^{-2-\epsilon}.$$
 By Theorem \ref{local-geometry},
 $$|\mathbf{T}|=O(r^{-2-\epsilon}), |\mathbf{\nabla T}|=O(r^{-3-\epsilon}), |\mathbf{\nabla^2 T}|=O(r^{-4-\epsilon}),$$
 so $\mathbf{T}\in H^2_{\delta}$ for all $\delta<3+2\epsilon$.
 By Theorem \ref{estimate-of-g} and the interior $L^2$ estimate (c.f. Theorem 9.11 of \cite{GilbargTrudinger}), for any large enough $R$,
 $$\int_{[R+2,\infty)\times \mathbb{T}^3}|\mathbf{T}|^2\le \int_{[R,\infty)\times \mathbb{T}^3}(r-R)^{\epsilon}(1-\chi(r-R))^2|\mathbf{T}|^2$$
 $$\le C \int_{[R,\infty)\times \mathbb{T}^3}(r-R)^{\epsilon+4}|\mathbf{D}((1-\chi(r-R)) \mathbf{T})|^2$$
 $$\le C||\mathbf{T}||_{H^2([R+1,R+2]\times \mathbb{T}^3)}^2\le C\int_{[R,R+3]\times \mathbb{T}^3}|\mathbf{T}|^2.$$
 So $$\int_{[R,\infty)\times \mathbb{T}^3}|\mathbf{T}|^2 \ge (1+1/C)\int_{[R+3,\infty)\times \mathbb{T}^3}|\mathbf{T}|^2.$$
 In other words, the Riemannian curvature decays exponentially in $L^2$ sense. The improvement to $L^\infty$ bound is simply Gilbarg and Trudinger's Theorem 9.20 in \cite{GilbargTrudinger}.
 \end{proof}
 From this better control of curvature, the holonomy of the loops $\gamma_{r,i}$ in Theorem \ref{1d-tangent-cone} can be improved to $|\mathrm{hol-Id}|<Ce^{-\mu r}$. Therefore, we are able to prove the following theorem:

 \begin{theorem}
 For any ALH-non-splitting gravitational instanton $(M,g)$, there exist a positive number $\mu$, a compact subset $K\subset M$, and a diffeomorphism $\Phi:[R,\infty)\times\mathbb{T}^3\rightarrow M-K$ such that
 $$|\nabla^m(\Phi^*g-h)|_h\le C(m)e^{-\mu r}$$
 for any $m=0,1,2,...$, where $h=\mathrm{d}r^2\oplus h_1$ for some flat metric $h_1$ on $\mathbb{T}^3$.
 \label{exponential-asymptotics}
 \end{theorem}

\subsection{Elliptic estimates with exponential growth weights}
 In this subsection, we are trying to prove the elliptic estimates for weighted Hilbert spaces with exponential growth weights.

 We first look at the Laplacian operator on $\mathbb{T}^3=\mathbb{R}^3/\Lambda$. Define the dual lattice $\Lambda^*$ by
 $$\Lambda^*=\{\lambda\in \mathbb{R}^3|<\lambda,v>\in\mathbb{Z},\forall v\in \Lambda\}.$$
 Then $\Delta$ has eigenvalues $-4\pi^2|\lambda|^2$ with eigenvectors $e^{2\pi i<\lambda,\theta>}$ for all $\lambda\in\Lambda^*$.
 We call $\delta$ critical if $\delta=4\pi|\lambda|$ for some $\lambda\in\Lambda^*$. So Theorem \ref{estimate-on-Rd} is replaced by the following theorem on $[R,\infty)\times \mathbb{T}^3$.

 \begin{theorem}
 Suppose $f$ is a real smooth function on $[0,\infty)\times \mathbb{T}^3$ supported in $[R,R']\times \mathbb{T}^3$, $\delta$ is not critical. Then
 $$\int_{[0,\infty)\times \mathbb{T}^3}|f|^2e^{\delta r}\mathrm{dVol}<C\int_{[0,\infty)\times \mathbb{T}^3}|\Delta f|^2e^{\delta r}\mathrm{dVol}.$$
 \label{estimate-on-r-times-T3}
 \end{theorem}

 \begin{proof}
 We write $f$ in terms its Fourier series
 $$f\sim\sum_{\lambda\in\Lambda^*} f_{\lambda}(r)e^{2\pi i<\lambda,\theta>}.$$
 Then $$ \begin{array}{lcl}
  \Delta f & \sim & \sum_{\lambda\in\Lambda^*}(f''_\lambda(r)-4\pi^2|\lambda|^2f_\lambda(r))e^{2\pi i<\lambda,\theta>} \\
  & = & \sum_{\lambda\in\Lambda^*}(\frac{\mathrm{d}}{\mathrm{d} r}-2\pi|\lambda|)(\frac{\mathrm{d}}{\mathrm{d} r}+2\pi|\lambda|)f_\lambda(r)e^{2\pi i<\lambda,\theta>}.\end{array} $$
 This time the Hardy's inequality is
 $$\int_0^\infty g^2e^{\nu r}\mathrm{d}r\le\frac{4}{\nu^2}\int_0^\infty (g')^2e^{\nu r}\mathrm{d}r.$$
 Therefore $$\begin{array}{lcl} \int_0^\infty f_\lambda^2e^{\delta r}\mathrm{d}r & = & \int_0^\infty  (e^{2\pi|\lambda|r}f_\lambda)^2e^{(\delta-4\pi|\lambda|)r}\mathrm{d}r\\
 & \le & \frac{4}{(\delta-4\pi|\lambda|)^2}\int_0^\infty[(e^{2\pi|\lambda|r}f_\lambda)']^2
 e^{(\delta-4\pi|\lambda|)r}\mathrm{d}r\\ & = &
 \frac{4}{(\delta-4\pi|\lambda|)^2}\int_0^\infty[(\frac{\mathrm{d}}{\mathrm{d} r}+2\pi|\lambda|)f_\lambda]^2e^{\delta r}\mathrm{d}r\\
 &\le& \frac{16\int_0^\infty[(\frac{\mathrm{d}}{\mathrm{d} r}-2\pi|\lambda|)(\frac{\mathrm{d}}{\mathrm{d} r}+2\pi|\lambda|)f_\lambda(r)]^2e^{\delta r}\mathrm{d}r}{(\delta+4\pi|\lambda|)^2(\delta-4\pi|\lambda|)^2}.\end{array}$$
 So as long as $\delta$ is not critical, we are done.
 \end{proof}

 Now we define $\underline{L}^2_{\delta}$ by $$||\phi||_{\underline{L}^2_{\delta}}=\int_M|\phi|^2e^{\delta r}\mathrm{dVol},$$
 and $\underline{H}^2_{\delta}$ by $$||\phi||_{\underline{H}^2_{\delta}}=\sqrt{\int_M|\phi|^2e^{\delta r}\mathrm{dVol}+\int_M|\nabla\phi|^2e^{\delta r}\mathrm{dVol}
 +\int_M|\nabla^2\phi|^2e^{\delta r}\mathrm{dVol}}.$$
 Then Theorem \ref{solve-laplacian-equation} is replaced by

 \begin{theorem}
 Suppose $\delta$ is not critical and $\delta<0$. For any $\phi\in \underline{L}^2_{\delta}$, there exists a tensor $\psi\in \underline{H}^2_{\delta}$ such that $\Delta\psi=\phi$.
 \label{solve-laplacian-equation-for-ALH}
 \end{theorem}

\subsection{Holomorphic functions on ALH instantons}

 To go through all the steps in ALF and ALG cases, we first need to control the growth rate of harmonic functions:

 \begin{lemma}
 Suppose $(N,h)$ is a smooth manifold such that outside a compact set, it is exactly $[R,\infty)\times\mathbb{T}^3$ with flat metric. Then any smooth function $u$ on $N$ harmonic outside a large enough ball with at most exponential growth rate can be written as linear combinations of $1$, $r$, $e^{2\pi|\lambda|r}e^{2\pi i<\lambda,\theta>}$ and an exponential decay function , where $r$ and $\theta$ are the coordinate functions on $[R,\infty)\times\mathbb{T}^3$ pulled back by the diffeomorphism.
 \label{explicit-harmonic-function}
 \end{lemma}

 \begin{proof}
 Write $u$ as its Fourier series $\sum_{\lambda\in\Lambda^*} u_{\lambda}(r)e^{2\pi i<\lambda,\theta>}.$ Then $u_{\lambda}''=4\pi^2|\lambda|^2u_{\lambda}$. So
 $$u\sim a_0+b_0 r+\sum_{\lambda\in\Lambda^*-\{0\}}a_\lambda e^{2\pi|\lambda|r}e^{2\pi i<\lambda,\theta>}+
 \sum_{\lambda\in\Lambda^*-\{0\}}b_\lambda e^{-2\pi|\lambda|r}e^{2\pi i<\lambda,\theta>}.$$
 By Parserval's identity, the growth condition of $u$ implies that the first sum has finite terms. For the second sum $U$, Parseval's identity again implies that $\int_{[R,R+1]\times\mathbb{T}^3}|U|^2$ decay exponentially. By Theorem 9.20 of \cite{GilbargTrudinger}, $U$ also decay exponentially in $L^{\infty}$ sense.
 \end{proof}

 Now  we can still find the global holomorphic function on ALH-non-splitting instanton $(M,g)$

 \begin{theorem}
 In the ALH-non-splitting case, there exist $(a_1,a_2,a_3)\in \mathbb{S}^2$ and a global holomorphic function with respect to $a_1I+a_2J+a_3K$ on $M$ such that any far enough fiber is biholomorphic to a complex torus.
 \label{ALH}
 \end{theorem}

 \begin{proof}
 As before, let $[R,\infty)\times\mathbb{T}^3=\{(r,\theta)|r\ge R,\theta=(\theta_1,\theta_2,\theta_3)\in \mathbb{R}^3/\Lambda\}$. Let $\Lambda^*$ be the dual lattice. Choose $\lambda\in \Lambda^*-\{0\}$ with minimal length. Choose $(a_1,a_2,a_3)\in \mathbb{S}^2$ such that
 $$(a_1I^*+a_2J^*+a_3K^*)\mathrm{d}r
 =-\frac{\lambda_1\mathrm{d}\theta_1+\lambda_2\mathrm{d}\theta_2+\lambda_3\mathrm{d}\theta_3}{|\lambda|}.$$

 The function $e^{2\pi|\lambda|r}e^{2\pi i<\lambda,\theta>}$ is holomorphic with respect to $a_1I+a_2J+a_3K$. The growth rate of this function is exactly $O(e^{2\pi|\lambda|r})$.

 Now we pull back this function from $[R,\infty)\times\mathbb{T}^3$ to $M$, cut it off and fill in with 0 inside $K$, we obtain a function $f$ satisfying $$\bar\partial_g f=\phi=O(e^{(2\pi|\lambda|-\mu)r}),$$ where $\mu$ is the constant in Theorem \ref{exponential-decay-of-curvature}.
 So $\phi\in \underline{L}^2_{-2\delta}$ for any non-critical positive number $\delta\in(2\pi|\lambda|-\mu,2\pi|\lambda|)$. By Theorem \ref{solve-laplacian-equation-for-ALH},   there exists $\psi\in \underline{H}^2_{-2\delta}$ satisfying
 $$\phi=\Delta\psi=-(\bar\partial^*\bar\partial+\bar\partial\bar\partial^*)\psi$$
 in the distribution sense. Elliptic regularity implies that $\psi$ is a smooth $(0,1)$-form. As before,
 $\bar\partial\psi=\xi\omega^+$ is a harmonic (0,2)-form. So $\xi$ is a harmonic function of order $O(e^{\delta r})$.

 Now we use a cut-off function and the diffeomorphism to average $g$ and the pull back of $h$. We get a smooth metric $g'$ on $M$ which is identically the pull back of $h$ outside a very big ball. Now let $\nu$ be the inferior of positive $\nu'$ such that $\xi$ is bounded by $O(e^{\nu' r})$. If $\nu>0$, then $\Delta_{g'}\xi\in \underline{L}^2_{-2\delta'}$ for any positive $\nu>\delta'>\nu-\mu$. It follows that there exists a function in $\underline{L}^2_{-2\delta'}$ whose Laplacian $\Delta_{g'}$ is $\Delta_{g'}\xi$. The difference of those two functions is a $g'$-harmonic function. By Lemma \ref{explicit-harmonic-function}, it must have at most linear growth rate since the growth rate is below the first nonlinear harmonic function. It follows that $\xi$ must lie in $O(e^{\delta' r})$, a contradiction. So $\nu=0$. Therefore, $\xi$ is bounded by any exponential growth function.

 So $\Delta_{g'}\xi$ decay exponentially. In particular, it is in $L^2_{1-\epsilon}$. By Theorem \ref{solve-laplacian-equation}, we can find out a function in $H^2_{-3-\epsilon}$ whose $\Delta_{g'}$ is $\Delta_{g'}\xi$. Therefore, we know that $\xi$ is actually in $O(r^{1+\epsilon})$. Of course, $\bar\partial\psi=\xi\omega^+$
 has the same estimate.

 By Lemma \ref{weighted-L2-estimate}, the harmonic (0,1)-form $\bar\partial^* \bar\partial \psi=\bar\partial(f+\bar\partial^*\psi)$ is in $O(r^{\epsilon})$ and its covariant derivative is in $O(r^{-1+\epsilon})$. The Weitzenb\"ock formula implies that $-\nabla^*\nabla(\bar\partial^* \bar\partial \psi)=0$. Therefore $$\int_M|\nabla(\bar\partial^* \bar\partial \psi)|^2\chi\le \int_M|\bar\partial^* \bar\partial \psi||\nabla(\bar\partial^* \bar\partial \psi)||\nabla\chi|$$
 for any smooth compactly supported $\chi$. Let $\chi=\chi(r-R)$, the right hand side converges to 0. Therefore $\bar\partial^* \bar\partial \psi$ is a covariant constant (0,1)-form. If this form is non-zero, it would be invariant under the holonomy of any loop. However, elements in $\mathrm{SU}(2)$ have no fixed point except the identity matrix. So $(M,g)$ must have trivial holonomy. Therefore, it is $\mathbb{R}^{4-k}\times \mathbb{T}^k$ with flat metric. It is a contradiction with our non-splitting assumption. So actually $\bar\partial^* \bar\partial \psi$ is identically 0. $f+\bar\partial^*\psi$ is a global holomorphic function on $M$.
 \end{proof}

 \subsection{Compactification of ALG and ALH-non-splitting instantons}
 In Theorem \ref{ALG} and \ref{ALH}, we proved the existence of global holomorphic function $u$ in ALG and ALH-non-splitting cases such that any far enough fiber is biholomorphic to a complex torus. Notice that $\mathrm{d}u$ is never zero on far enough fiber. Define a holomorphic vector field $X$ by $\omega^+(Y,X)=\mathrm{d}u(Y)$. Then since $X(u)=\mathrm{d}u(X)=\omega^+(X,X)=0$, $X$ is well defined when it is restricted to each far enough fiber. On each fixed far enough fiber, there exists a unique holomorphic form $\phi$ such that $\phi(X)=1$. Locally
 $$\omega^+=f(u,v)\mathrm{d}u\wedge\mathrm{d}v, X=f^{-1}(u,v)\frac{\partial}{\partial v}, \phi=f(u,v)\mathrm{d}v.$$

 Notice that each far enough fiber is topologically a torus. So actually, we can integrate the from $\phi$ to get a holomorphic function $v \in \mathbb{C}/\mathbb{Z}\tau_1(u)\oplus\mathbb{Z}\tau_2(u)$ up to a constant.
 We can fix this constant locally by choosing a holomorphic section of $u$ as the base point.
 Therefore $M$ is biholomorphic to $U\times\mathbb{C}/(u,v)\sim(u,v+m\tau_1(u)+n\tau_2(u))$, where $\tau_1(u)$ and $\tau_2(u)$ are locally defined holomorphic functions. Actually, they are the integral of $\phi$ in the basis of $H_1$ of each fiber. This gives a holomorphic torus fiberation locally.

 Recall that there is a diffeomorphism from $M$ minus a large compact set to the standard fiberation. Denote the inverse image of the zero section by $s$. $s$ is again a section outside large compact set because $\mathrm{d}u$ differ with the standard one by a decaying error. Write $\bar\partial s$ as $e(u)\mathrm{d}\bar{u}\otimes X$, then $e$ is a function defined on the inverse of the punctured disc with polynomial growth rate. So there is an at most polynomial growth function $E$  on the inverse of punctured disc such that $\bar\partial E(u)=e(u)\mathrm{d}\bar{u}$. Now we apply the flow $-E(u)X$ to the section $s$ to get a holomorphic section $s_0$ on the neighborhood of infinity. View $s_0$ as the zero section, we know that $M$ minus a large compact set is biholomorphic to $(\mathbb{C}-B_R)\times \mathbb{C}/(u,v)\sim(u,v+m\tau_1(u)+n\tau_2(u))$ globally, where $\tau_1(u)$ and $\tau_2(u)$ are multi-valued holomorphic functions.

 As proved in Kodaira's paper \cite{Kodaira}, there exists an (unique) elliptic fiberation $B$ over $B_{R^{-1}}$ with a section such that $B$ minus the central fiber $D$ is biholomorphic to $(B_{R^{-1}}-\{0\})\times \mathbb{C}/(\tilde u,v)\sim(\tilde u,v+m\tau_1(\tilde u^{-1})+n\tau_2(\tilde u^{-1}))$. We can naturally identity points and get a compactification $\bar M$ of $M$. So $\bar M$ is a compact complex surface with a meromorphic function $u=\tilde u^{-1}$. Now since the subvariety of critical points $\{\mathrm{d}u=0\}$ is a finite union of irreducible curves (On those irreducible curves, $u$ is of course constant) and points, we know that except finite critical values in $\mathbb{CP}^1$, any fiber of $u$ has no intersection with  $\{\mathrm{d}u=0\}$. Therefore, a generic fiber has genus 1 and must be an elliptic curve. In other words, $\bar M$ is a compact elliptic surface. In conclusion, we have proved the second main theorem.

 \subsection{Twistor space of ALF-$D_k$ instantons}

 On ALF-$D_k$ gravitational instantons, we have found quadratic growth holomorphic functions for each compatible complex structure. A natural question is, is there any relationship between those functions? Before going ahead, let us recall the definition of twistor space of hyperk\"ahler manifolds.

 \begin{definition}
 (c.f. \cite{HithcinKarlhedLindstromRocek}) Let $(M,g,I,J,K)$ be a hyperk\"ahler manifold. Then the twistor space $Z$ of $M$ is the product manifold $M\times \mathbb{S}^2$ equipped with an integrable complex structure
 $$\underline{I}=(\frac{1-\zeta\bar\zeta}{1+\zeta\bar\zeta}I-\frac{\zeta+\bar\zeta}{1+\zeta\bar\zeta}J
 +i\frac{\zeta-\bar\zeta}{1+\zeta\bar\zeta}K,I_0),$$
 where $\zeta\in\mathbb{C}\subset\mathbb{C}\cup\{\infty\}=\mathbb{CP}^1=\mathbb{S}^2$ is the coordinate function, and $I_0$ is the standard complex structure on $\mathbb{CP}^1$.
 \end{definition}

 Notice that our definition is different from \cite{HithcinKarlhedLindstromRocek} to correct a sign error. We will briefly rewrite Page 554-557 of their paper with a correct sign.

 Let $\phi$ be a (1,0)-form of $I$. Then $I^*\phi=i\phi$, where $(I^*\phi)(X)=\phi(IX)$. Set $\theta=\phi+\zeta K^*\phi$,
 then $$(1+\zeta\bar\zeta)\underline{I}^*\theta=
 ((1-\zeta\bar\zeta)I^*-(\zeta+\bar\zeta)J^*+i(\zeta-\bar\zeta)K^*)\theta=i(1+\zeta\bar\zeta)\theta,$$
 because we have relationships like $J^*I^*=K^*$. (In \cite{HithcinKarlhedLindstromRocek}, they thought $I^*J^*=K^*$ and caused a sign error.)

 Now if the form $\omega^+=\omega_2+i\omega_3$ can be written as
 $$\frac{1}{2}\omega^+=\sum_{i=1}^{n} \phi_i\wedge\phi_{n+i}$$
 for some (1,0)-forms of $I$. Then we can define a form on the twistor space by
 $$\omega=2\sum_{i=1}^{n} (\phi_i+\zeta K^*\phi_i)\wedge(\phi_{n+i}+\zeta K^*\phi_{n+i})= (\omega_2+i\omega_3)+2\zeta\omega_1-\zeta^2(\omega_2-i\omega_3).$$
 It is a holomorphic section of the vector bundle $\Lambda^2T_F^*\otimes \mathcal{O}(2)$, where $F$ means the fiber of $Z$ which is diffeomorphic to $M$.
 We also have a real structure $\tau(p,\zeta)=(p,-1/\bar\zeta)$. It takes the complex structure $\underline{I}$ to its conjugate $-\underline{I}$. In \cite{HithcinKarlhedLindstromRocek}, they proved the following theorem:

 \begin{theorem}
 Let $Z^{2n+1}$ be a complex manifold such that

 (i) $Z$ is a holomorphic fiber bundle $\pi:Z\rightarrow \mathbb{CP}^1$ over the projective line;

 (ii) The bundle admits a family of holomorphic sections each with normal bundle isomorphic to $\mathbb{C}^{2n}\otimes \mathcal{O}(1)$;

 (iii) There exists a holomorphic section $\omega$ of $\Lambda^2T_F^*\otimes \mathcal{O}(2)$ defining a symplectic form on each fiber;

 (iv) $Z$ has a real structure compatible with (i),(ii),(iii) and inducing the antipodal map on $\mathbb{CP}^1$.

 Then the parameter space of real sections is a $4n$-dimensional manifold with a natural hyperk\"ahler metric for which $Z$ is the twistor space.
 \label{twistor-space}
 \end{theorem}

 Return to the gravitational instantons for which $n$ in the above theorems equals to 1. Recall that we have found the holomorphic function on $M$ by modifying the pull back of the standard function on the standard model. So let us look at the standard model $(E,h,I,J,K)$ first. It is the quotient of the Taub-NUT metric outside a compact set by $D_{4|e|}$. Recall that the Taub-NUT metric is (c.f. Theorem \ref{standard-fiberation})
 $$\mathrm{d}s^2=V\mathrm{d}\mathbf{x}^2+V^{-1}\eta^2$$
 with $$\mathrm{d}x_1=I^*(V^{-1}\eta)=J^*\mathrm{d}x_2=K^*\mathrm{d}x_3.$$
 So $$(\frac{1-\zeta\bar\zeta}{1+\zeta\bar\zeta}I^*-\frac{\zeta+\bar\zeta}{1+\zeta\bar\zeta}J^*
 +i\frac{\zeta-\bar\zeta}{1+\zeta\bar\zeta}K^*-i)(-2\zeta\mathrm{d}x_1-(1-\zeta^2)\mathrm{d}x_3+i(1+\zeta^2)\mathrm{d}x_2)=0.$$
 Therefore, $(-x_3+ix_2-2x_1\zeta-(-x_3-ix_2)\zeta^2)^2$ is a holomorphic function on the twistor space of $E$.
 So the holomorphic function on $M\times\{\zeta\}\in Z$ is asymptotic to $(-x_3+ix_2-2x_1\zeta-(-x_3-ix_2)\zeta^2)^2$ with error $O'(r^{2-\epsilon})$

 Notice that any harmonic function has even integer growth rate, so the holomorphic function is unique up to the adding of constant. We may fix this ambiguity by requiring the value at the fixed base point $o$ to be 0. We will prove that after the modification the holomorphic functions have some simple relationship.

 Actually, we have a $I$-holomorphic ($\zeta=0$) function $u_1+iv_1$ asymptotic to $(-x_3+ix_2)^2=(x_3^2-x_2^2)-2ix_2x_3$, $J$-holomorphic ($\zeta=-1$) function $u_2+iv_2$ asymptotic to $(2x_1+2ix_2)^2=4(x_1^2-x_2^2)+8ix_1x_2$, and $K$-holomorphic  ($\zeta=-i$) function $u_3+iv_3$ asymptotic to $(-2x_3+2ix_1)^2=4(x_3^2-x_1^2)-8ix_3x_1$. Notice that $u_2+u_3-4u_1$ is a harmonic function asymptotic to 0, i.e. in $O'(r^{2-\epsilon})$, so it must be 0. Similarly the harmonic function
 $$z(p,\zeta)=(u_1+iv_1)-\frac{1}{2}(v_3+iv_2)\zeta+\frac{1}{2}(u_2-u_3)\zeta^2+\frac{1}{2}(v_3-iv_2)\zeta^3+(u_1-iv_1)\zeta^4$$
 is asymptotic to $(-x_3+ix_2-2x_1\zeta-(-x_3-ix_2)\zeta^2)^2$ and therefore must be the holomorphic one. In conclusion, we have proved the following theorem:

 \begin{theorem}
 In the ALF-$D_k$ case, there exist 6 harmonic functions $u_i$,$v_i$ with $4u_1=u_2+u_3$ such that $$z(p,\zeta)=(u_1+iv_1)-\frac{1}{2}(v_3+iv_2)\zeta+\frac{1}{2}(u_2-u_3)\zeta^2+\frac{1}{2}(v_3-iv_2)\zeta^3+(u_1-iv_1)\zeta^4$$ is a $\underline{I}$-holomorphic map from the twistor space of $M$ to the total space of the $\mathcal{O}(4)$ bundle over $\mathbb{CP}^1$.
 \label{holomorphic-map-on-twistor-space}
 \end{theorem}

 There is a real structure on the $\mathcal{O}(4)$ bundle $(\zeta,\eta) \rightarrow (-1/\bar\zeta,\bar\eta/\bar\zeta^{4})$.
 It is easy to see that the map $z$ commutes with the real structure.


\begin{thebibliography}{99}
 \bibitem{AndersonKronheimerLebrun} Anderson,Michael.T.; Kronheimer,P.B.; LeBrun,C.: Complete Ricci-flat K\"ahler manifolds of infinite topological type. Comm. Math. Phys. 125 (1989), no.4, 637-642. MR1024931.
 \bibitem{AtiyahHitchin} Atiyah,M.; Hitchin,N.J.: The geometry and dynamics of magnetic monopoles. M. B. Porter Lectures. Princeton University Press, Princeton, NJ, 1988. viii+134 pp. MR0934202.
 \bibitem{BandoKasueNakajima}  Bando,S.; Kasue,A.; Nakajima,H.: On a construction of coordinate at infinity on manifolds with fast curvature decay and maximal volume growth. Invent. Math. 97 (1989), no.2, 313-349. MR1001844.
 \bibitem{BiquardMinerbe} Biquard,O.; Minerbe,V.: A Kummer construction for gravitational instantons. Comm. Math. Phys.  308 (2011), no.3, 773-794. MR2855540.
 \bibitem{BuserKarcher} Buser,P.; Karcher,H.: Gromov's almost flat manifolds. Ast\'erisque, 81. Soci\'et\'e Math\'ematique de France, Paris, 1981. MR0619537.
 \bibitem{CheegerColding} Cheeger,J.; Colding,T.H.: Lower bounds on Ricci curvature and the almost rigidity of warped products. Ann. of Math. (2) 144 (1996), no.1, 189-237. MR1405949.
 \bibitem{CheegerFukayaGromov} Cheeger,J.; Fukaya,K.; Gromov,M.: Nilpotent structures and invariant metrics on collapsed manifolds. J. Amer. Math. Soc. 5 (1992), no.2, 327-372. MR1126118.
 \bibitem{CheegerGromoll} Cheeger,J; Gromoll,D: The splitting theorem for manifolds of nonnegative Ricci curvature. J. Differential Geometry 6 (1971/72), 119-128. MR0303460.
 \bibitem{CherkisHitchin} Cherkis,S.A.; Hitchin,N.J.: Gravitational instantons of type $D_k$. Comm. Math. Phys. 260 (2005), no.2, 299-317. MR2177322.
 \bibitem{CherkisKapustin} Cherkis,S.A.; Kapustin,A.: Singular monopoles and gravitational instantons. Comm. Math. Phys.  203 (1999), no.3, 713-728. MR1700937.
 \bibitem{Colding} Colding,T.H.: Ricci curvature and volume convergence. Ann. of Math. (2) 145 (1997), no.3, 477-501. MR1454700.
 \bibitem{Drees} Drees,G.: Asymptotically flat manifolds of nonnegative curvature. Differential Geom. Appl. 4 (1994), no.1, 77-90. MR1264910.
 \bibitem{Fukaya} Fukaya,K.: A boundary of the set of the Riemannian manifolds with bounded curvatures and diameters. J. Differential Geom. 28 (1988), no.1, 1-21. MR0950552.
 \bibitem{GilbargTrudinger} Gilbarg,D.; Trudinger,N.S.: Elliptic partial differential equations of second order. Reprint of the 1998 edition, Springer-Verlag, Berlin, 2001. MR1814364.
 \bibitem{GreeneWu} Greene, R. E.; Wu, H.: Function theory on manifolds which possess a pole. Lecture Notes in Mathematics, 699. Springer, Berlin, 1979. MR0521983.
 \bibitem{Gromov} Gromov,M.: Almost flat manifolds. J. Differential Geom. 13 (1978), no.2, 231-241. MR0540942.
 \bibitem{HaskinsHeinNordstrom} Haskins,M.; Hein,H.-J.; Nordstr\"om,J.: Asymptotically cylindrical Calabi-Yau manifolds. J. Differential Geom. 101 (2015), no. 2, 213-265. MR3399097.
 \bibitem{HauselHunsickerMazzeo} Hausel,T.; Hunsicker,E.; Mazzeo,R.: Hodge cohomology of gravitational instantons. Duke Math. J. 122 (2004), no.3, 485-548. MR2057017.
 \bibitem{Hawking} Hawking,S.: Gravitational instantons. Phys. Lett. 60A (1977) 81-83. MR0465052.
 \bibitem{Hein}  Hein,H.-J.: Gravitational instantons from rational elliptic surfaces. J. Amer. Math. Soc. 25(2012), no.2, 355-393. MR2869021.
 \bibitem{HithcinKarlhedLindstromRocek} Hitchin,N.J.; Karlhede, A.; Lindstr\"om, U.; Ro\v{c}ek, M.: Hyper-K\"ahler metrics and supersymmetry. Comm. Math. Phys. 108 (1987), no.4, 535-589. MR0877637.
 \bibitem{Hormander} H\"ormander,L: An introduction to complex analysis in several variables. Third edition. North-Holland Publishing Co., Amsterdam, 1990. MR1045639.
 \bibitem{IvanovRocek} Ivanov,I.T.; Ro\v{c}ek,M.: Supersymmetric $\sigma$-models, twistors, and the Atiyah-Hitchin metric. Comm. Math. Phys. 182 (1996), no.2, 291-302. MR1447294.
 \bibitem{JostKarcher} Jost,J\"urgen; Karcher,Hermann: Geometrische Methoden zur Gewinnung von a-priori-Schranken f\"ur harmonische Abbildungen. Manuscripta Math.  40 (1982), no.1, 27-77. MR0679120.
 \bibitem{Kasue} Kasue,A.: A compactification of a manifold with asymptotically nonnegative curvature. Ann. Sci. \'Ecole Norm. Sup. (4) 21 (1988), no.4, 593-622. MR0982335.
 \bibitem{Kodaira} Kodaira,K.: On compact analytic surface II. Ann. of Math. (2) 77 (1963), 563-626. MR0184257.
 \bibitem{Kronheimer1} Kronheimer,P.B.: The construction of ALE spaces as hyper-K\"ahler quotients. J. Differential Geom. 29 (1989), no.3, 665-683. MR0992334.
 \bibitem{Kronheimer2} Kronheimer,P.B.: A Torelli-type theorem for gravitational instantons. J. Differential Geom. 29 (1989), no.3, 685-697. MR0992335.
 \bibitem{LeBrun} LeBrun,C.: Complete Ricci-flat K\"ahler metrics on $C^n$ need not be flat. Several complex variables and complex geometry, Part 2 (Santa Cruz, CA, 1989),  297-304, Proc. Sympos. Pure Math., 52, Part 2, Amer. Math. Soc., Providence, RI, 1991. MR1128554.
 \bibitem{MashikoNaganoOtsuka} Mashiko,Y.; Nagano,K.; Otsuka,K.: The asymptotic cones of manifolds of roughly non-negative radial curvature. J. Math. Soc. Japan 57 (2005), no.1, 55-68. MR2114720.
 \bibitem{MinerbeMass} Minerbe,V.: A mass for ALF manifolds. Comm. Math. Phys. 289 (2009), no.3,925-955. MR2511656.
 \bibitem{MinerbeAsymptotic} Minerbe,V.: On the asymptotic geometry of gravitational instantons. Ann. Sci. \'Ec. Norm. Sup\'er. (4)43(2010), no.6, 883-924. MR2778451.
 \bibitem{MinerbeMultiTaubNUT} Minerbe,V.: Rigidity for multi-Taub-NUT metrics. J. Reine Angew. Math. 656 (2011), 47-58. MR2818855.
 \bibitem{Ruh} Ruh,Ernst A.: Almost flat manifolds. J. Differential Geom. 17 (1982), no.1, 1-14. MR0658470.
 \bibitem{TianYau} Tian,G.; Yau,S.-T.: Complete K\"ahler manifolds with zero Ricci curvature. I. J. Amer. Math. Soc. 3  (1990),  no.3, 579-609. MR1040196.
 \bibitem{Yau} Yau,S.T.: The role of partial differential equations in differential geometry. Proceedings of the International Congress of Mathematicians (Helsinki, 1978), pp.237-250, Acad. Sci. Fennica, Helsinki, 1980. MR0562611.
 \end{thebibliography}
 \end{document}